\documentclass{conm-p-l}

\copyrightinfo{2008}{American Mathematical Society}

\usepackage{amsmath}
\usepackage{amssymb}
\usepackage[all]{xy}

\numberwithin{equation}{section}

\newtheorem{thm}{Theorem}[section]
\newtheorem{lem}[thm]{Lemma}
\newtheorem{prop}[thm]{Proposition}
\newtheorem{cor}[thm]{Corollary}
\newtheorem{rem}[thm]{Remark}
\newtheorem{exam}[thm]{Example}
\newtheorem{exam-nota}[thm]{Example-Notation}
\newtheorem{nota}[thm]{Notation}
\newtheorem{dfn}[thm]{Definition}

\newtheorem{dfn-nota}[thm]{Definition-Notation}

\newtheorem{dfn-lem}[thm]{Lemma-Definition}
\newtheorem{dfn-thm}[thm]{Theorem-Definition}
\theoremstyle{definition}
\newtheorem*{ack}{Acknowledgements}

\newcommand{\beqa}{\begin{eqnarray*}}
\newcommand{\eeqa}{\end{eqnarray*}}

\newcommand{\fa}{\mbox{${\mathfrak a}$}}

\newcommand{\fg}{\mbox{${\mathfrak g}$}}

\newcommand{\fh}{\mbox{${\mathfrak h}$}}

\newcommand{\fz}{\mbox{${\mathfrak z}$}}
\newcommand{\fm}{\mbox{${\mathfrak m}$}}

\newcommand{\C}{\mbox{${\mathbb C}$}}

\newcommand{\Ad}{{\rm Ad}}

\newcommand{\ad}{\text{ad}}

\newcommand{\Co}{\mathbb{C}}

\newcommand{\orbx}{\mathcal{O}_{x}}
\newcommand{\fgl}{\mathfrak{gl}}
\newcommand{\fso}{\mathfrak{so}}

\newcommand{\xifij}{\xi_{f_{i,j}}}
\newcommand{\sreg}{\fg_{n}^{sreg}}
\newcommand{\fgn}{\fg_{n}}
\newcommand{\fibre}{(\mathfrak{g}_{n})_{c}}
\newcommand{\sfibre}{(\mathfrak{g}_{n})_{c}^{sreg}}
\newcommand{\fgno}{(\mathfrak{g}_{n})_{\Omega}}

\newcommand{\Gprod}{Z_{GL(1)}(J_{1})\times\cdots\times Z_{GL(n-1)}(J_{n-1})}
\newcommand{\sol}{\Xi^{i}_{c_{i}, \, c_{i+1}}}

\newcommand{\glfibre}{\fgl(n)_{c}}
\newcommand{\glsfibre}{\fgl(n)_{c}^{sreg}}
\newcommand{\ofibre}{\fso(n)_{c}}

\newcommand{\glnil}{\fgl(n)_{0}}
\newcommand{\glsnil}{\fgl(n)_{0}^{sreg}}
\newcommand{\soll}{\Xi^{2l+1}_{c_{2l+1}, c_{2l+2}}}

\begin{document}

\title[]{The Gelfand-Zeitlin Integrable System and Its Action on Generic Elements of $\fgl(n)$ and $\fso(n)$}

\author[M. Colarusso]{Mark Colarusso}
\address{Department of Mathematics, University of Notre Dame, Notre Dame, IN, 46556}
\email{mcolarus@nd.edu}
\maketitle

\begin{abstract}
In recent work Bertram Kostant and Nolan Wallach ([KW1], [KW2]) have defined an interesting action of a simply connected Lie group $A\simeq \mathbb{C}^{{n\choose 2}}$ on $\fgl(n)$ using a completely integrable system derived from Gelfand-Zeitlin theory.   In this paper we show that an analogous action of $\Co^{d}$ exists on the complex orthogonal Lie algebra $\fso(n)$, where $d$ is half the dimension of a regular adjoint orbit in $\fso(n)$.  In [KW1],  Kostant and Wallach describe the orbits of $A$ on a certain Zariski open subset of regular semisimple elements in $\fgl(n)$.  We extend these results to the case of $\fso(n)$.  We also make brief mention of the author's results in [Col1], which describe all $A$-orbits of dimension ${n\choose 2}$ in $\fgl(n)$.  
\end{abstract}

\section{Introduction}\label{s:intro}
Let $\fg_{n}$ be the complex general linear Lie algebra $\fgl(n, \Co)$ or the complex orthogonal Lie algebra $\fso(n, \Co)$.  We think of $\fso(n)$ as the Lie algebra of $n\times n$ complex skew-symmetric matrices.  Let $d$ be half the dimension of a regular adjoint orbit in $\fg_{n}$.  In this paper, we describe the construction of an analytic action of $\Co^{d}$ on $\fgn$ using a Lie algebra of commutative vector fields derived from Gelfand-Zeitlin theory.  We then describe the action of $\Co^{d}$ on a Zariski open subset of regular semisimple elements in $\fg_{n}$.  For the case of $\fgn=\fgl(n)$, these results were proven in recent work of Kostant and Wallach in [KW1].  In the case of $\fgn=\fso(n)$, these results are new.  They first appeared in the author's doctoral thesis [Col].  

The paper is structured as follows.  In section \ref{s:intro}, we give an exposition of the work of Kostant and Wallach in [KW1].  We indicate how their results generalize to the case of $\fso(n)$ providing new proofs where necessary.   In section \ref{s:generics}, we describe the action of the group $\Co^{d}$ on a certain Zariski open subset of regular semisimple elements in $\fgn$.  The results for the case of $\fgn=\fgl(n)$ are contained in Theorems 3.23 and 3.28 in [KW1].  In section \ref{s:glngenerics}, we indicate a different proof of these results, which more readily generalizes to the case of $\fgn=\fso(n)$.  The proof in section \ref{s:glngenerics} is taken from some preliminary work of Kostant and Wallach.  For the case of $\fso(n)$, we give complete proofs of the analogues of Theorems 3.23 and 3.28 in [KW1] in section \ref{s:orthogen}.  Section \ref{s:summary} summarizes some of the other main results of the author's doctoral thesis without proof.  These results are to appear in an upcoming publication [Col1].  

We now briefly summarize the main results of each section.  To construct the action of $\Co^{d}$ on $\fgn$, we consider the Lie-Poisson structure on $\fgn\simeq\fgn^{*}$.  Let $\fg_{i}=\fgl(i),\text{ or } \fso(i)$ for $1\leq i\leq n$.  Then $\fg_{i}\subset\fgn$ is a subalgebra, where we think of an $i\times i$ matrix as the top left hand corner of an $n\times n$ matrix with all other entries zero.  Let $P(\fg_{i})$ be the algebra of polynomials on $\fg_{i}$.  Any polynomial $f\in P(\fg_{i})$ defines a polynomial on $\fgn$ as follows. For $x\in\fgn$, let $x_{i}$ denote the $i\times i$ submatrix in the top left hand corner of $x$.  Then one can show that $f(x)=f(x_{i})$.  Let $P(\fg_{i})^{G_{i}}=\Co[f_{i,1},\cdots ,f_{i, r_{i}}], \, r_{i}=rank(\fg_{i})$ denote the ring of $\Ad$-invariant polynomials on $\fg_{i}$.  In section \ref{s:GZa}, we will see that the functions $\{ f_{i,j} | 1\leq i\leq n, 1\leq j\leq r_{i}\}$ Poisson commute and in section \ref{s:generics}, we will show that they are algebraically independent.  These observations along with the surprising fact that the sum


\begin{equation}\label{eq:sumranks}
\sum_{i=1}^{n-1} r_{i}=d
\end{equation}
gives us that the functions $\{f_{i,j} | 1\leq i\leq n-1, \, 1\leq j\leq r_{i}\}$ form a completely integrable system on certain regular adjoint orbits.    

\begin{rem}
Note for $\fg_{n}=\fso(n)$, $\fg_{1}=\fso(1)=0$, so that $r_{1}=0$.  Thus, the first function in the collection $\{f_{i,1}, \cdots , f_{i, r_{i}}, \,1\leq i\leq n,\, 1\leq j\leq r_{i}\}$  is $f_{2,1}$.  We will retain this convention throughout the paper.
\end{rem}
We make a choice of generators for the ring of $\Ad$-invariant polynomials $P(\fg_{i})^{G_{i}}$.  If $\fg_{i}=\fgl(i)$, we take as generators
\begin{equation}\label{eq:ainvars}
f_{i,j}(x)=tr(x_{i}^{j}) \text{ for } 1\leq i\leq n \text{ and } 1\leq j\leq i.
\end{equation}
 For $\fg_{i}=\fso(i)$, we have to consider two cases.  If $\fg_{i}=\fso(2l) $ is of type $D_{l}$, we take 
 \begin{equation}\label{eq:Dinvars}
  f_{i,j}(x)=tr(x_{i}^{2j})\text{ for } 1\leq j\leq l-1
\text { and }  f_{i,l}(x)=Pfaff(x_{i}),
\end{equation} where $Pfaff(x_{i})$ denotes the Pfaffian of $x_{i}$.  If $\fg_{i}=\mathfrak{so}(2l+1,\mathbb{C})$ is of type $B_{l}$, we take  
\begin{equation}\label{eq:Binvars}
f_{i, j}(x)=tr(x_{i}^{2j})\text{ for } 1\leq j\leq l.
\end{equation}

Let $\fa$ be the Lie algebra of vector fields on $\fgn$ generated by the Hamiltonian vector fields $\xifij$ for the functions $\{f_{i,j}| \,  1\leq i\leq n-1, \, 1\leq j\leq r_{i}\}$.  In [KW1], it is shown that $\fa$ integrates to an action of $\Co^{d}=\Co^{n\choose 2}$ on $\fgn=\fgl(n)$.  The following theorem appears in section \ref{s:GZg} where we give a general proof that also covers the case of $\fgn=\fso(n)$.  
\begin{thm}\label{thm:iGZgroup} 
The Lie algebra $\fa$ integrates to an analytic action of $A=\Co^{d}$ on $\fg_{n}$.    
\end{thm} 
We call this group $A=\Co^{d}$ following the notation of [KW1].

We call an element $x\in\fg_{n}$ strongly regular if its orbit under the group $A$ of Theorem \ref{thm:iGZgroup} is of maximal dimension $d$.  It is not difficult to see that if $x$ is strongly regular, then $x$ is regular, and its $A$-orbit is a Lagrangian submanifold of the adjoint orbit containing $x$.  (See Proposition \ref{prop:lag} and Remark \ref{r:sreg} in section \ref{s:sreg}.)  

In section \ref{s:generics}, we describe the $A$-orbit structure of a Zariksi open set of regular semisimple elements defined by 
$$
\fgno=\{x\in\fgn |\;x_{i}\text { is regular semisimple}, \, \sigma(x_{i-1})\cap \sigma (x_{i})=\emptyset, \, 2\leq i\leq n-1\},
$$
where for $y\in\fg_{i}$, $\sigma(y)$ denotes the spectrum of $y$ regarded as an element of $\fg_{i}$.  To study the action of $A$ on $\fgno$, it is helpful to study the action of $A$ on a certain class of fibres of the corresponding moment map.  We denote the moment map by $\Phi:\fgn\to \Co^{d+r_{n}}$,
\begin{equation}\label{eq:imom}
\Phi(x)=(f_{1,1}(x_{1}), f_{2,1}(x_{2}),\cdots, f_{n, r_{n}}(x))
\end{equation}
for $x\in\fgn$.  For $c\in \Co^{d+r_{n}}$, we denote the fibres of $\Phi$ by $\Phi^{-1}(c)=\fibre$.  To define these special fibres, we consider a Cartan subalgebra $\fh_{i}\subset \fg_{i}$, and we let $W_{i}$ be the Weyl group with respect to $\fh_{i}$.  We can identify the orbit space $\fh_{i}/W_{i}$ with $\Co^{r_{i}}$ via the map
\begin{equation}\label{eq:worb}
[h]_{W_{i}}\to (f_{i,1}(h),\cdots , f_{i, r_{i}}(h)),
\end{equation}
where $[h]_{W_{i}}$ denotes the $W_{i}$ orbit of $h\in\fh_{i}$.  Using this identification, we can think of the moment map in (\ref{eq:imom}) as a map $\fgn\to \fh_{1}/W_{1}\times\cdots \times \fh_{n}/W_{n}$.  We define $\Omega_{n}\subset \fh_{1}/W_{1}\times\cdots \times \fh_{n}/W_{n}$ to be the subset of $c=(c_{1},\cdots, c_{n})\in\fh_{1}/W_{1}\times\cdots \times \fh_{n}/W_{n}$ with the property that $c_{i}\in \fh_{i}/W_{i}$ is regular and the elements in the orbits $c_{i}$ and $c_{i+1}$ have no eigenvalues in common.  We can understand the action of $A$ on $\fgno$ by analyzing its action on the fibres $\fibre$ for $c\in \Omega_{n}$.  The main theorem concerning the orbit structure of the set $\fgno$ is Theorem \ref{thm:generics}, which is given in section \ref{s:generics}.
\begin{thm} \label{thm:igenerics}
The elements of $\fgno$ are strongly regular.  For $c\in\Omega_{n}$, the fibre $\fibre$ is precisely one orbit under the action of the group $A$ given in Theorem \ref{thm:iGZgroup}.  Moreover, $\fibre$ is a homogenous space for a free, algebraic action of the $d$-dimensional torus $(\Co^{\times})^{d}$.  
\end{thm}

\begin{ack}
The author would like to thank Nolan Wallach for his guidance and assistance as a thesis supervisor.
\end{ack}


\section{The Gelfand-Zeitlin Integrable System}\label{s:intro}
\subsection{The Lie-Poisson structure on $\fg$} \label{s:LP}

	We first consider a general setting.  Let $\fg$ be a finite dimensional, reductive Lie algebra over $\Co$.   Let $\beta(\cdot, \cdot)$ be the $G$-invariant form on $\fg$. 
Then $\fg$ is a Poisson manifold.  We now describe the Poisson structure.   First, we need a few preliminary notions.  Let $\mathcal{H}(\fg)$ denote the set of holomorphic functions on $\fg$.  For $x,\, y\in\fg$ define $\partial_{x}^{y}\in T_{x}(\fg)$ to be the directional derivative in the direction of $y$ evaluated at $x$ (i.e. $\partial_{x}^{y}f= \frac{d}{dt} |_{t=0}\, f(x+ty) $ for $f\in\mathcal{H}(\fg)$).   Let $\psi\in\mathcal{H}(\fg)$ and $x\in \fg$, then $d\psi_{x}\in T_{x}^{*}(\fg)\simeq \fg^{*}$.  Using the form $\beta$, we can naturally identify $d\psi_{x}$ with an element of $\fg$ denoted by $\nabla\psi (x)$ defined by the rule
\begin{equation}\label{eq:rule}
\beta(\nabla\psi(x), z)=\frac{d}{dt} |_{t=0}\, \psi(x+tz)=(\partial_{x}^{z}\psi)
\end{equation}
for all $z\in\fg$.   Then, if $\{f, g\}$ denotes the Poisson bracket of any two functions $f,\, g\in\mathcal{H}(\fg)$, one can show (see [CG], pg 36)
	\begin{equation}\label{eq:pois}
	\{f,g\}(x)=\beta(x, [\nabla f(x), \nabla g(x)]).
	\end{equation}

Note that if $f,g \in\mathcal{H}(\fg)$, then (\ref{eq:pois}) implies that their Poisson bracket $\{f, g\}\in\mathcal{H}(\fg)$.  Hence, $\mathcal{H}(\fg)$ is a Poisson algebra.  Using the form $\beta$, we can identify this Poisson structure on $\fg$ with the Lie-Poisson structure on $\fg^{*}$.  The Lie-Poisson structure on $\fg^{*}$ is the unique Poisson structure on $\fg^{*}$ such that the Poisson bracket of linear functions $f, \, g\in(\fg^{*})^{*}=\fg$ is the Lie bracket of $f$, $g$ (i.e. $\{f, g\}=[f,g]$) [CG].  In particular, the symplectic leaves are the adjoint orbits of the adjoint group $G$ of $\fg$ [Va].  Let $x\in\fg$ and $\orbx$ be its adjoint orbit.  The symplectic structure on $\orbx$ is often referred to as the Kostant-Kirillov-Souriau (KKS) structure. (See [CG, pg 23] for an explicit description of this structure.)

	
	
For each $f\in\mathcal{H}(\fg)$ we define a Hamiltonian vector field $\xi_{f}$.  The action of $\xi_{f}$ on $\mathcal{H}(\fg)$ is $\xi_{f}(g)=\{f, g\}$.  Using, (\ref{eq:pois}) we can compute the Hamiltonian vector field at a point $x\in\fg$, 
\begin{equation}\label{eq:tanvec}
(\xi_{f})_{x}=\partial_{x}^{[x,\nabla f(x)]}.
\end{equation} 
With this description of $(\xi_{f})_{x}$ it is easy to see that $(\xi_{f})_{x}\in T_{x}(\orbx)$.  

Our work focuses on adjoint orbits $\orbx$ of maximal dimension.  For $x\in\fg$, let $\fz_{\fg}(x)$ be the centralizer of $x$ in $\fg$.  
An element $x\in\fg$ is said to be regular if $\dim \fz_{\fg}(x)=r$, where $r$ is the rank of $\fg$.  Thus, $x$ is regular if and only if $\dim \fz_{\fg}(x)$ is minimal.  This is equivalent to $\dim\orbx=\dim \fg-r$ being maximal. 

We are interested in constructing polarizations of regular adjoint orbits.   A polarization of a symplectic manifold $(M,\omega)$ is an integrable subbundle $P\subset TM$ which is Lagrangian i.e. $P_{m}=P_{m}^{\perp}$ for all $m\in M$.   
It then follows that $\dim(P_{m})=\frac{1}{2}\dim (M)=d$ for $m\in M$.  Suppose that $f_{1},\cdots, f_{d}$ are independent Poisson commuting functions on $M$.  Independent means that the differentials of these functions $\{ df_{i} |\, 1\leq i\leq d\}$ are linearly independent on an open, dense subset of $M$ (see [C]). (If $f_{1}, \cdots , f_{d}$ are polynomials and $M$ is a smooth affine variety, then this definition is equivalent to the statement that $f_{1},\cdots, f_{d}$ are algebraically independent.)  The span $\{ \xi_{f_{i}} |1\leq i\leq d\}$ of the Hamiltonian vector fields gives a polarization on an open, dense subset of $M$.  The integral submanifolds of this polarization are necessarily Lagrangian submanifolds of $M$.  The functions $f_{1},\cdots, f_{d}$ are often referred to as a (completely) integrable system [C].  
In the case $(M,\omega)=(\orbx, \omega)$, where $\omega$ is the KKS symplectic structure and $\orbx$ is a regular adjoint orbit, we want to find $d$ independent Poisson commuting functions where $2d=\dim\orbx=\dim\fg-r$.  If $\fg=\fgl(n)$ or $\fso(n)$, we will produce such a family using a classical analogue of the Gelfand-Zeitlin algebra in the polynomials on $\fg$, $P(\fg)$.  

\subsection{A classical analogue of the Gelfand-Zeitlin algebra}\label{s:GZa}

For the remainder of the paper let $\fg_{n}=\fgl(n),\, \fso(n)$.  We represent $\fso(n)$ as $n\times n$ complex skew-symmetric matrices.   We can take the form $\beta$ of the last section to be the trace form.  Let $\fg_{i}=\fgl(i), \, \fso(i)$ for $1\leq i\leq n$.  Let $G_{i}$ be the corresponding adjoint group.   We then have a natural inclusion of subalgebras $\mathfrak{g}_{i}\hookrightarrow \fg_{n}$.  The embedding is
$$
Y\hookrightarrow\left [\begin{array}{cc} Y & 0\\
0 & 0\end{array}\right ]_{\mbox{\large ,}}$$
which puts the $i\times i$ matrix $Y$ as the top left hand corner of an $n\times n$ matrix.  We also have a corresponding embedding of the adjoint groups 
$$
g\rightarrow \left [\begin{array}{cc} g & 0\\
0 & Id_{n-i}\end{array}\right ]_{\mbox{\large ,}}
$$
where $Id_{n-i}$ is the $(n-i)\times (n-i)$ identity matrix.   We always think of $\fg_{i}\hookrightarrow \fg_{n}$ and $G_{i}\hookrightarrow G_{n}$ via these two embeddings unless otherwise stated. 
  We make the following definition.
\begin{dfn}\label{d:cutoffs}
For $x\in\fg_{n}$, let $x_{i}\in\fg_{i}$ be the top left hand corner of $x$, i.e. $(x_{i})_{k,l}=x_{k,l}$ for $1\leq k, l\leq i$.  We refer to $x_{i}$ as the $i\times i$ cutoff of $x$.
\end{dfn}

The set of polynomials $P(\fg_{n})$ on $\fg_{n}$ is a Poisson subalgebra of $\mathcal{H}(\fg_{n})$ (see (\ref{eq:pois})).  For any $i$, $1\leq i\leq n$,
\begin{equation}\label{eq:sum} 
\fgn=\fg_{i}\oplus\fg_{i}^{\perp},
\end{equation}
 where $\fg_{i}^{\perp}$ denotes the orthogonal complement of $\fg_{i}$ with respect to the trace form on $\fgn$.  Thus, we can use the trace form on $\fgn$ to identify $\fg_{i}\simeq \fg_{i}^{*}$.  This implies $P(\fg_{i})\subset P(\fg_{n})$ is a Poisson subalgebra.  Explicitly, if $f\in P(\fg_{i})$ and $x\in\fg_{n}$, then $f(x)=f(x_{i})$.  Moreover, the Poisson structure on $P(\fg_{i})$ inherited from $P(\fg_{n})$ agrees with the Lie-Poisson structure on $P(\fg_{i})$ [KW1, pg 330].  Thus, the $\Ad$-invariant polynomials on $\fg_{i}$, $P(\fg_{i})^{G_{i}}$, are in the Poisson centre of $P(\fg_{i})$, since their restriction to any adjoint orbit of $G_{i}$ in $\fg_{i}$ is constant.  Hence, the subalgebra of $P(\fg_{n})$ generated by the different $\Ad$-invariant polynomial rings $P(\fg_{i})^{G_{i}}$ for all $i$, $1\leq i\leq n$ is Poisson commutative.  We refer to this algebra as $J(\fg_{n})$.

  %

\begin{equation}\label{eq:invaralg}
J(\fg_{n})=P(\fg_{1})^{G_{1}}\otimes\cdots\otimes P(\fg_{n})^{G_{n}}.
\end{equation}
We say that the Poisson commutative algebra $J(\fg_{n})$ is a classical analogue of the Gelfand-Zeitlin algebra in $P(\fg_{n})$.  The Gelfand-Zeitlin algebra $GZ(\fg_{n})$ is the associative subalgebra of the universal enveloping algebra of $\fg_{n}$, $U(\fg_{n})$, generated by the centres $Z(\fg_{i})$ of $U(\fg_{i})$ for $1\leq i\leq n$,  i.e. $GZ(\fg_{n})=Z(\fg_{1})\cdots Z(\fg_{n})$.  The isomorphism  $Z(\fg_{i})\simeq S(\fg_{i})^{G_{i}}$ (see Theorem 10.4.5 in [Dix]) then justifies our terminology, because $S(\fg_{i})^{G_{i}}\simeq P(\fg_{i}^{*})^{G_{i}}\simeq P(\fg_{i})^{G_{i}}$.  From now on we simply refer to $J(\fg_{n})$ as the Gelfand-Zeitlin algebra.

\begin{rem}\label{r:GZ}
The Gelfand-Zeitlin algebra is a polynomial algebra in ${n+1\choose 2}$ generators ([DFO]).   We will soon see that this is also true of the algebra $J(\fg_{n})$ (see section \ref{s:generics}), and therefore $J(\fgn)\simeq GZ(\fg_{n})$ as associative algebras.  
\end{rem}


Since $J(\fg_{n})$ is Poisson commutative,  $V=\{\xi_{f} |f\in J(\fg_{n})\}$ is a commutative Lie algebra of Hamiltonian vector fields.    We define a general distribution by 
\begin{equation}\label{eq:distr}
V_{x}=span\{(\xi_{f})_{x} |f\in J(\fg_{n})\}\subset T_{x}(\fg_{n}).
\end{equation}
We observe that if $\{f_{i}\}_{i\in I}$ generate the Gelfand-Zeitlin algebra $J(\fg_{n})$, then $V_{x}=span\{(\xi_{f})_{x} |f\in J(\fg_{n})\}=span\{(\xi_{f_{i}})_{x} | i\in I\}$.  This follows directly from the Leibniz rule, which implies that $df\in span \{df_{i}\}_{i\in I}$.  Let $f_{i,1},\cdots, f_{i, r_{i}}$, with $r_{i}=rank(\fg_{i})$ generate the ring $P(\fg_{i})^{G_{i}}$.  $J(\fg_{n})$ is then generated by the polynomials $f_{i,1},\cdots, f_{i, r_{i}}$ for $1\leq i\leq n$.  Recall that if $f\in P(\fg_{n})^{G_{n}}$, then $\xi_{f}=0$.  Thus, 
\begin{equation}
\dim V_{x}\leq \sum_{i=1}^{n-1} r_{i}.
\end{equation}
 For $\fg_{n}=\fgl(n),\, \fso(n)$ we compute
\begin{equation}\label{eq:countd}
\sum_{i=1}^{n-1} r_{i}=d,
\end{equation}
where $d$ is half of the dimension of a regular adjoint orbit $\orbx$ in $\fgn$.  If we can show that the functions $f_{i,1}, \cdots , f_{i, r_{i}}$, $1\leq i\leq n$ are algebraically independent, we will have a completely integrable system on certain regular adjoint orbits.  

In the next section, we show that the Hamiltonian vector field $\xifij$ of $f_{i,j}$ is complete.  Since the vector fields $\xifij$ commute for all $i, j$, we obtain a global action of $\Co^{d}$ on $\fgn$.  Thus, we can study the Gelfand-Zeitlin system by studying the action of $\Co^{d}$ on $\fgn$.  In section \ref{s:sreg}, we show that the existence of orbits of $\Co^{d}$ of dimension $d$ is equivalent to the algebraic independence of the functions $\{f_{i,j}| 1\leq i\leq n,\, 1\leq j\leq r_{i}\}$.  We will then see that the $\Co^{d}$ orbits of dimension $d$ are Lagrangian submanifolds of certain regular adjoint orbits.  In section \ref{s:generics}, we describe examples of such $\Co^{d}$ orbits and obtain the complete integrability of the Gelfand-Zeitlin system on certain regular semisimple adjoint orbits.  



\begin{rem}
The difficulty of trying to reproduce this scheme for the symplectic Lie algebra $\mathfrak{sp}(2n, \Co)$ is that $\sum_{i=1}^{n-1} r_{i}< \frac{1}{2} \dim\orbx$, $\orbx$ a regular adjoint orbit in $\mathfrak{sp}(2n, \Co)$.  Thus, no choice of subalgebra of $J(\fg_{n})$ gives rise to a completely integrable system.  One can check that one needs an extra $n^{2}-\frac{n(n-1)}{2}$ independent functions. 
\end{rem}

\subsection{The group $A$ } \label{s:GZg}

Let $f\in J(\fg_{n})$.  The remarkable fact about the Hamiltonian vector fields $\xi_{f}$ is that they are complete.  We first discuss a special case of this fact.  Let $r_{i}=rank(\fg_{i})$ and let $\{ f_{i,j} | 1\leq i\leq n,\, 1\leq j\leq r_{i}\}$ be as in the previous section.  The vector field $\xifij$ integrates to a global action of $\Co$ on $\fgn$ for each $i$, $j$. 
\begin{thm}\label{thm:GZgroup}
Let $d$ be half the dimension of a regular adjoint orbit in $\fgn$.  Let $\fa$ be the commutative Lie algebra generated by the vector fields $\{\xi_{f_{i,j}} | 1\leq i\leq n-1, \, 1\leq j\leq r_{i}\}$.  Then $\fa$ integrates to an action of $\Co^{d}$ on $\fgn$.  The orbits of $\Co^{d}$ are leaves of the distribution $x\to V_{x}$ given by (\ref{eq:distr}).  The action of $\Co^{d}$ stabilizes adjoint orbits.  
\end{thm} 
\begin{proof}
By equation (\ref{eq:tanvec}), we have $(\xi_{f_{i,j}})_{x}=\partial_{x}^{[-\nabla f_{i,j}(x),x]}$.  The key observation is that for any $\phi\in P(\fg_{i})^{G_{i}}$, $y\in\fg_{i}$, $\nabla \phi(y)\in\fz_{\fg_{i}}(y)$, where $\fz_{\fg_{i}}(y)$ denotes the centralizer of $y$ in $\fg_{i}$.  We readily note that $\nabla f_{i,j}(x)=\nabla f_{i,j}(x_{i})\in\fg_{i}$ for $x\in\fgn$.  Thus, $\nabla f_{i,j}(x)\in\fz_{\fg_{i}}(x_{i})$.  Using this fact, we can show  
\begin{equation}\label{eq:action}
\theta(t,\, x)=\Ad(\exp(-t\, \nabla f_{i,j}(x_{i})))\cdot x
\end{equation} 
is the integral curve for the vector field $\xifij$ starting at $x\in\fg_{n}$.  We compute the differential to the curve $\theta(t,x)$ at an arbitrary $t_{0}\in\Co$. 
$$ 
\begin{array}{cc}
\frac{d}{dt} |_{t=t_{0}} \,\Ad(
\exp(-t\, \nabla f_{i,j}(x_{i})))\cdot x  = &\\
\\
\frac{d}{dt} |_{t=t_{0}}\exp(t \,\ad( -\nabla f_{i,j}(x_{i})))\cdot x =&\\
\\
\ad(-\nabla f_{i,j}(x_{i}))\cdot (\exp(t_{0} \,\ad\,( -\nabla f_{i,j}(x_{i})))\cdot x). & 
\end{array}
$$
We let $$y=\exp(t_{0} \,\ad \,(-\nabla f_{i,j}(x_{i})))\cdot x=\Ad(\exp(-t_{0}\, \nabla f_{i,j}(x_{i})))\cdot x=\theta(t_{0},\,x).$$  Since $-\nabla f_{i,j}(x_{i})$ centralizes $x_{i}$,  $\theta(t,\,x)_{i}=x_{i}$ for all $t\in\C$.  In particular, we have $y_{i}=x_{i}$.  This implies $$\ad(-\nabla f_{i,j}(x_{i}))\cdot\, (\exp(-t_{0}\,\ad(\nabla f_{i,j}(x_{i})))\cdot x)=\ad(-\nabla f_{i,j}(y_{i})) \cdot y=(\xifij)_{y}$$ by equation (\ref{eq:tanvec}), which verifies the claim.  
To complete the proof of the theorem, we observe that since the Lie algebra $\fa$ is commutative, the flows of the vector fields $\xifij$ all commute.  Thus, the actions of $\Co$ in (\ref{eq:action}) commute and give rise to an action of $\Co^{d}$ on $\fg_{n}$.  It follows easily from (\ref{eq:action}) that this action of $\Co^{d}$ preserves the adjoint orbits.  
\end{proof}

The proof given here is the one in [Col].  For a different proof in the case of $\fgn=\fgl(n)$, see Theorems 3.3, 3.4 in [KW1].  Using the the completeness of the vector fields $\xifij$, one can then prove the completeness of any Hamiltonian vector field $\xi_{f}$ for $f\in J(\fgn)$.  One can also show that the foliation of $\fg_{n}$ given by the action of $\Co^{d}$ in Theorem \ref{thm:GZgroup} is independent of the choice of generators for the Gelfand-Zeitlin algebra $J(\fg_{n})$.  
\begin{thm}\label{thm:integrate}
Let $f\in J(\fg_{n})$.  The Hamiltonian vector field $\xi_{f}$ integrates to a global action of $\Co$ on $\fg_{n}$.  Suppose that the polynomials $\{q_{i} | 1\leq i\leq k\}$ generate the Gelfand-Zeitlin algebra.  Let $\fa^{\prime}$ be the Lie algebra generated by the Hamiltonian vector fields $\{ \xi_{q_{i}} | 1\leq i\leq k\}$.  Then $\fa^{\prime}$ integrates to an action of $\Co^{k}$ on $\fgn$.  This action commutes with the action of $\Co^{d}$ of Theorem \ref{thm:GZgroup}.  The orbits of the action of $\Co^{k}$ on $\fgn$ are the same as the action of $\Co^{d}$ in Theorem \ref{thm:GZgroup}.   \end{thm}
For a proof, we refer the reader to Theorem 3.5 in [KW1].  The proof given there works in the orthogonal case without modification.

Since we are concerned with the geometry of orbits of the Gelfand-Zeitlin system of maximal dimension $d$, we loose no information in studying a specific action of $\Co^{d}$ on $\fg_{n}$ by fixing a choice of generators $\{f_{i,j} |1\leq i\leq n, 1\leq j\leq r_{i}\}$ for $J(\fg_{n})$.  For $\fg_{i}=\fgl(i)$, we take the generators for $P(\fg_{i})^{G_{i}}$ to be given by equation (\ref{eq:ainvars}).  For $\fg_{n}=\fso(n)$, we have to consider two cases.  If $\fg_{i}=\fso(2l) $ is of type $D_{l}$, we take the generators in (\ref{eq:Dinvars}).  If $\fg_{i}=\mathfrak{so}(2l+1,\mathbb{C})$ is of type $B_{l}$, we take the generators in (\ref{eq:Binvars}).  
 %

\begin{nota}
Let $\fgn=\fgl(n)$, $f_{i,j}(x)=tr(x_{i}^{j})$, and $\fa$ be as in Theorem \ref{thm:GZgroup}.  Kostant and Wallach refer to the unique simply connected Lie group with Lie algebra $\fa$ as $A\simeq \Co^{{n\choose 2}}=\Co^{d}$.  The group $A$ acts on $\fgl(n)$ via the action of $\Co^{{n\choose 2}}$ in Theorem \ref{thm:GZgroup}, see [KW1] Theorem 3.3.  We adopt this terminology for both $\fgl(n)$ and $\fso(n)$.  That is to say, for $\fgn=\fso(n)$ we will refer to the group $\Co^{d}$ as $A$ and the action of $\Co^{d}$ given in Theorem \ref{thm:GZgroup} as the action of $A$.  
\end{nota}

It is illustrative to write out the vector fields for the Lie algebra $\fa$ in the case of $\fgl(n)$ (see Theorem 2.12 in [KW1]).  
\begin{equation}\label{eq:coords}
 (\xi_{f_{i,j}})_{x}=\partial_{x}^{[-jx_{i}^{j-1}, x]}.
\end{equation}
 This follows from the fact that for $f_{i,j}=tr(x_{i}^{j})$, $\nabla f_{i,j}(x)=j x_{i}^{j-1}$.  Using (\ref{eq:action}) and (\ref{eq:coords}), we see that $\xi_{f_{i,j}}$ integrates to an action of $\mathbb{C}$ on $\fgl(n)$ given by
 \begin{equation}\label{eq:flows}
 \Ad\left (\left [\begin{array}{cc} \exp(tjx_{i}^{j-1}) & 0 \\
0 & Id_{n-i} \end{array}\right ]\right)\cdot x
\end{equation}
for $t\in\mathbb{C}$.  
The orbits of $A$ are then the composition of the flows in (\ref{eq:flows}) for $1\leq i\leq n-1$, $1\leq j\leq i$ in any order.  

Unfortunately, for the case of $\fso(n)$, the $A$-orbits do not have such a clean description.  However, we can say that they are given by composing the flows in (\ref{eq:action}) in any order.

Using (\ref{eq:coords}), we get a fairly easy description of the distribution $V_{x}$ defined in (\ref{eq:distr}) for $x\in\fgl(n)$.  We define $Z_{x}=\sum_{i=1}^{n-1} Z_{x_{i}}$, where $Z_{x_{i}}$ is the associative subalgebra of $\fgl(i)\hookrightarrow \fgl(n)$ generated by $Id_{i}$ and $x_{i}$.  From (\ref{eq:coords}), it follows that 
\begin{equation}\label{eq:easydist}
V_{x}=span\{\partial_{x}^{ [z,x]} | z\in Z_{x}\}.
\end{equation}
 
 In the case of $\fg_{n}=\fso(n)$, there is no simple description of the distribution $V_{x}$ as in (\ref{eq:easydist}).  The difficulty lies in the fact that the differential of the Pfaffian is not a power of the matrix.  However, for our purposes it will suffice to describe $V_{x}$ as
 \begin{equation}\label{eq:odist}
 V_{x}=span\{ \partial_{x}^{[\nabla f_{i,j}(x_{i}), x]} | \, 2\leq i\leq n-1, \, 1\leq j\leq r_{i}\}.
 \end{equation}


\subsection{Strongly regular elements and the polarization of adjoint orbits} \label{s:sreg}
The results of this section are taken from [KW1] unless otherwise stated.  We provide proofs that are valid for both $\fgl(n)$ and $\fso(n)$ for the convenience of the reader.  With the exception of Proposition \ref{prop:constr}, the proofs presented are the ones in [KW1], which automatically generalize to the case of $\fso(n)$.  For Proposition \ref{prop:constr}, we present a different proof, which easily incorporates both $\fgl(n)$ and $\fso(n)$. 

	Let $r_{i}=rank(\fg_{i})$.  In this section, we show that the algebraic independence of the functions $\{f_{i,j}(x) | 1\leq i\leq n, \,1\leq j\leq r_{i}\}$ is equivalent to the existence of orbits of the group $A$ of maximal dimension $d$.  In section \ref{s:generics}, we will produce such orbits using a special Zariski open subset of regular semisimple elements in $\fgn$.  We accordingly make the following theorem-defintion.


 \begin{dfn-thm}\label{d:sreg}
$x\in\fg_{n}$ is said to be strongly regular if and only if the differentials $\{(df_{i,j})_{x} | 1\leq i\leq n, 1\leq i \leq r_{i}\}$ are linearly independent at $x$.  This is equivalent to the $A$-orbit of $x$ being of maximal dimension $d$.   We denote the set of strongly regular elements of $\fg_{n}$ by $\fg_{n}^{sreg}$. 
\end{dfn-thm}
Before giving a proof of this fact, we have to recall a basic result of Kostant (see [K, pg 382]).   

\begin{thm}\label{thm:reg}
 Let $x$ be an element of a reductive Lie algebra $\fg$.  Then $x$ is regular if and only if $(d\phi_{1})_{x}\wedge\cdots\wedge (d\phi_{l})_{x}\neq 0$, where $\phi_{1},\cdots, \phi_{l}$ generate the ring $P(\fg)^{G}$.  
\end{thm}

\begin{proof}
Suppose $x\in\sreg$.  Then the differentials $(df_{i,j})_{x}$ are linearly independent at $x$.  Let $q_{i,j}=f_{i,j} |_{\orbx}$, with $\orbx$ the adjoint orbit containing $x$.  To show that the $A$-orbit of $x$ is of dimension $d$, it suffices to show that the tangent vectors $(\xifij)_{x}\in T_{x}(\orbx)$ for $1\leq i\leq n-1, \, 1\leq j\leq r_{i}$ are linearly independent.  This follows from the penultimate statement in Theorem \ref{thm:GZgroup}.  Because $\orbx$ is symplectic, the independence of the tangent vectors  $(\xifij)_{x}\in T_{x}(\orbx)$ is equivalent to the independence of the differentials $\{(dq_{i,j})_{x},\, 1\leq i\leq n-1, 1\leq j\leq r_{i}\}$.  Suppose to the contrary that the differentials $\{(dq_{i,j})_{x},\, 1\leq i\leq n-1, 1\leq j\leq r_{i}\}$ are linearly dependent.  This implies that there exist constants $c_{i,j}, 1\leq i\leq n-1, 1\leq j\leq r_{i}$ not all $0$ such that the sum $\sum_{1\leq i\leq n-1, 1\leq j\leq r_{i}} c_{i,j}\,(df_{i,j})_{x}\in T_{x}(\orbx)^{\perp}$, where $T_{x}(\orbx)^{\perp}\subset T_{x}^{*}(\fg_{n})$ is the annhilator of $T_{x}(\orbx)$.  Since $x$ is strongly regular, the differentials $\{(df_{n,j})_{x}, 1\leq j\leq r_{n}\}$ are independent, so by Theorem \ref{thm:reg}, $x\in\fg_{n}$ is regular.  It follows that the set $\{(df_{n,j})_{x}, 1\leq j\leq r_{n}\}$ forms a basis of $ T_{x}(\orbx)^{\perp}$.  But this implies the existence of a non-trivial linear combination of the differentials $\{(df_{i,j})_{x},\, 1\leq i\leq n-1, 1\leq j\leq r_{i}\}$ and the differentials $\{(df_{n,j})_{x},\, 1\leq j\leq r_{n}\}$, contradicting the fact that $x\in\sreg$.  

Now, suppose that the $A$-orbit through $x$ has dimension $d$.  This is equivalent to the tangent vectors $\{(\xifij)_{x}| \; 1\leq i\leq n-1, 1\leq j\leq r_{i}\}$ being linearly independent in $T_{x}(\orbx)$.  The Poisson commutativity of the functions $f_{i,j}$ gives that $V_{x}$ is an isotropic subspace of the symplectic vector space $T_{x}(\orbx)$.  It follows that $\dim\orbx= \dim V_{x}+\dim V_{x}^{\perp}\geq 2d$.  We recall that $2d$ is the maximal dimension of an adjoint orbit in $\fg_{n}$, and therefore the inequality is forced to be equality and $x$ is regular.  By Theorem \ref{thm:reg} the differentials $\{(df_{n,j})_{x},\, 1\leq j\leq r_{n}\}$ are linearly independent.  Thus, the differentials $\{(df_{i,j})_{x},\, 1\leq i\leq n-1, 1\leq j\leq r_{i}\}$ and $\{(df_{n,j})_{x},\, 1\leq j\leq r_{n}\}$ are linearly independent.  It follows easily that the differentials $\{(df_{i,j})_{x},\, 1\leq i\leq n, 1\leq j\leq r_{i}\}$ are independent, and therefore $x$ is strongly regular.  
\end{proof}

Using the penultimate statement in Theorem \ref{thm:GZgroup} and Theorem-Definition \ref{r:sreg}, we obtain 
\begin{equation}\label{eq:sreg}
x\in \fgn^{sreg}\Leftrightarrow \dim V_{x}=d,
\end{equation}
where $V_{x}\subset T_{x}(\fgn)$ is as in (\ref{eq:easydist}) and (\ref{eq:odist}).  

 The connection between polarizations of regular adjoint orbits and $\fgn^{sreg}$ is contained in the following proposition.
 \begin{prop}\label{prop:lag}
 
Let $\orbx$ be the adjoint orbit containing $x\in\fgn$.  Let $\orbx^{sreg}=\orbx\cap\sreg$.  If $\orbx^{sreg}\neq\emptyset$, then $x$ is regular.  In this case $\orbx^{sreg}$ is  Zariski open in $\orbx$ and is therefore a symplectic manifold.  Moreover, $\orbx^{sreg}$ is a union of $A$-orbits of dimension $d=\frac{1}{2}\dim \orbx$, which are necessarily Lagrangian submanifolds of $\orbx^{sreg}$.  Thus, the $A$-orbits in $\orbx^{sreg}$ form the leaves of a polarization on $\orbx^{sreg}$.  
\end{prop}
\begin{proof}
If $\mathcal{O}_{x}^{sreg}$ is non-empty, then by Proposition \ref{prop:sreg} it is clear that $x$ is regular. 
For $y\in\mathcal{O}_{x}^{sreg}$, $T_{y}(A\cdot y)$ is Lagrangian, since it is isotropic and of dimension exactly half the dimension of the ambient manifold $\mathcal{O}_{x}^{sreg}$.  Thus, the $A$-orbits in $\mathcal{O}_{x}^{sreg}$ are Lagrangian submanifolds of $\mathcal{O}_{x}^{sreg}$, and we have our desired polarization.
\end{proof}

\begin{rem}\label{r:sreg}
The corresponding result in [KW1] is stronger than the result stated here.  It also states that if $x$ is regular in $\fgl(n)$, then $\orbx^{sreg}$ is non-empty.  Thus, $\fgl(n)^{sreg}$ is non-empty, and any regular adjoint orbit in $\fgl(n)$ possesses a dense, open submanifold which is foliated by Lagrangian submanifolds.  However, it is not clear that the same result holds in the case of $\fso(n)$.  In section \ref{s:orthogen}, we will construct polarizations of certain regular semisimple adjoint orbits in $\fso(n)$.  
\end{rem} 
 
We now give a more concrete characterization of strongly regular elements.  
\begin{prop}\label{prop:sreg}
Let $x\in\fg_{n}$ and let $\fz_{\fg_{i}}(x_{i})$ denote the centralizer in $\fg_{i}$ of $x_{i}$.  Then $x$ is strongly regular if and only if the following two conditions hold.
\begin{itemize}
\item (a) $x_{i}\in\fg_{i}$ is regular for all $i$, $1\leq i\leq n$.  
\item (b) $\fz_{\fg_{i}}(x_{i})\cap \fz_{\fg_{i+1}}(x_{i+1})=0$ for all $1\leq i\leq n-1$.
\end{itemize}
\end{prop}
We will make use of only part of this proposition, namely that if $x\in\fg_{n}$ is strongly regular, then $x_{i}$ is regular for all $i$.  However, we prove the proposition in its entirety for completeness.  

\begin{proof}
Suppose that $x\in\sreg$, then by Theorem-Definition \ref{d:sreg} the differentials $\{(df_{i,j})_{x},\, 1\leq i\leq n, 1\leq j\leq r_{i}\}$ are linearly independent.  In particular for each $i, \, 1\leq i\leq n$ the differentials $\{(df_{i,j})_{x},\, 1\leq j\leq r_{i}\}$ are independent, which implies that $x_{i}$ is regular for all $i$ by Theorem \ref{thm:reg}.  The elements $\{\nabla f_{i,j}(x),\, 1\leq j\leq r_{i}\}$ then form a basis for the centralizer $\fz_{\fg_{i}}(x_{i})$.  The linear independence of the elements $\{\nabla f_{i,j}(x),\, 1\leq i\leq n, 1\leq j\leq r_{i}\}$ then implies the sum $\sum_{i=1}^{n} \fz_{\fg_{i}}(x_{i})$ is direct, which implies (b).  
  
  Now, suppose that both (a) and (b) hold.  We claim (b) implies that the sum 
  \begin{equation}\label{eq:centsum}
  \sum_{i=1}^{n} \fz_{\fg_{i}}(x_{i})
  \end{equation}
   is direct.  Suppose to the contrary that we have an increasing sequence $\{1\leq i_{1}  < \cdots < i_{m}\leq n\}$ and elements $z_{i_{j}}\neq 0\in\fz_{\fg_{i_{j}}}(x_{i_{j}})$ with the property that 
  \begin{equation}\label{eq:firstcentralizer}
 \sum_{j=1}^{m} z_{i_{j}}=0.
  \end{equation}
We claim this forces
\begin{equation}\label{eq:part}
[z_{i_{j}} , x]_{i_{1}+1}=0
\end{equation}
for $j>1$.  To see this, we make use of the decomposition $\fgn=\fg_{i_{j}}\oplus \fg_{i_{j}}^{\perp}$ (see (\ref{eq:sum})).  The component of $x$ in $\fg_{i_{j}}$ is clearly $x_{i_{j}}$, but $[ z_{i_{j}} , x_{i_{j}}]=0$.  Since $\ad\, \fg_{i_{j}} $ stabilizes the components of the above decomposition, we have $[z_{i_{j}}, x]\in \fg_{i_{j}}^{\perp}$.    
Now, since $i_{j}\geq i_{1}+1$, we have $\fg_{i_{j}}^{\perp}\subseteq \fg_{i_{1}+1}^{\perp}$, yielding equation (\ref{eq:part}).  Equations (\ref{eq:firstcentralizer}) and (\ref{eq:part}), then imply 
$$
[z_{i_{1}}, x]_{i_{1}+1}=0.
$$
But $z_{i_{1}} \in \fg_{i_{1}}$, and therefore
\begin{equation}\label{eq:etc}
[z_{i_{1}} , x_{i_{1}+1}]=[z_{i_{1}}, x]_{i_{1}+1}=0.
\end{equation}
Thus, $z_{i_{1}}\in\fz_{\fg_{i_{1}}}(x_{i_{1}})\cap \fz_{\fg_{i_{1}+1}}(x_{i_{1}+1})=0$, which is a contradiction.  From (a) and Theorem \ref{thm:reg}, it follows that the differentials $\{(d f_{i,j})_{x},\, 1\leq j\leq r_{i}\}$ are linearly independent for each $i$, $1\leq i\leq n$.   The fact the sum in (\ref{eq:centsum}) is direct then implies the entire set of differentials $\{(d f_{i,j})_{x},\, 1\leq i\leq n, \,1\leq j\leq r_{i}\}$ is linearly independent.  Thus, $x$ is strongly regular.  


  
\end{proof}

We conclude this section with a technical result about strongly regular orbits that will be of use to us in section \ref{ss:moment}. 
\begin{prop}\label{prop:constr}
Let $x\in\sreg$.  Let $Z_{G_{i}}(x_{i})$ denote the centralizer in $G_{i}$ of $x_{i}$.   Consider the morphism of affine algebraic varieties
$$
\begin{array}{c}
\psi: Z_{G_{1}}(x_{1})\times Z_{G_{2}}(x_{2})\times\cdots\times Z_{G_{n-1}}(x_{n-1})\to \fg_{n},\\
\\
\psi(g_{1},\cdots ,g_{n-1})=\Ad(g_{1})\Ad(g_{2})\cdots \Ad(g_{n-1})\cdot x.
\end{array}
$$
 The image of $\psi$ is exactly the $A$-orbit of $x$, $A\cdot x$.  Hence $A\cdot x$ is an irreducible, Zariski constructible subset of $\fgn$.  
\end{prop}
\begin{proof}
We first show $A\cdot x\subset Im\psi$, where $Im\psi$ denotes the image of the morphism $\psi$.  Let $x\in\fgn$.  For $\underline{t}\in A=\Co^{d}$, we write $\underline{t}=(t_{1,1},\cdots, t_{i,j},\cdots,  t_{n-1, r_{n-1}})$ with $t_{i,j}\in\Co$.  In these coordinates, the action of the $(i,j)$-th coordinate $v_{i,j}=(0,\cdots, t_{i,j}, 0, \cdots, 0)$ on $x$ is given by the flow of $\xifij$ as in (\ref{eq:action})
$$
\theta(t_{i,j},\, x)=\Ad(\exp(-t_{i,j}\, \nabla f_{i,j}(x_{i})))\cdot x.
$$
We noted in the proof of Theorem \ref{thm:GZgroup} that this action of $\Co$ centralizes $x_{i}$.  Thus, the action of $\underline{t}$ on $x$ is
\begin{equation}\label{eq:orbit}
\underline{t}\cdot x=\Ad(\exp(-t_{1,1}\, \nabla f_{1,1}(x_{1})))\cdots \Ad(\exp(-t_{n-1, r_{n-1}}\, \nabla f_{n-1, r_{n-1}}(x_{n-1})))\cdot x.
\end{equation} 
The expression in (\ref{eq:orbit}) is in $Im\psi$, because $\nabla f_{i,j}(x_{i})\in\fz_{\fg_{i}}(x_{i})$, and therefore $\exp(c\nabla f_{i,j}(x_{i}))\in Z_{G_{i}}(x_{i})$ for any $c\in \Co$.  

We now prove $Im\psi\subset A\cdot x$.  To show this inclusion, we make use of the fact that $x$ is strongly regular.  By Proposition \ref{prop:sreg}, $x_{i}$ is regular for all $i$.  A basic result of Kostant (Proposition 14 in [K]) says that $Z_{G_{i}}(x_{i})$ is an abelian, connected algebraic group.  Since $Z_{G_{i}}(x_{i})$ is a connected algebraic group over $\Co$, it is also connected as a complex Lie group (see Theorem 11.1.22 in [GW]).  Thus, given $g_{i}\in Z_{G_{i}}(x_{i})$, $g_{i}=\exp (V)$ with $V\in Lie(Z_{G_{i}}(x_{i}))=\fz_{\fg_{i}}(x_{i})$.  Since $x_{i}\in\fg_{i}$ is regular, $V=\sum_{j=1}^{r_{i}} c_{i,j} \nabla f_{i,j}(x_{i})$ with $c_{i,j}\in\Co$ by Theorem \ref{thm:reg}.  This implies $g_{i}=\exp(c_{i,1} \nabla f_{i,1}(x_{i}))\cdots \exp(c_{i,r_{i}}\nabla f_{i, r_{i}} (x_{i})) $.  Repeating this argument for each $g_{i}$ for $1\leq i\leq n-1$ and using (\ref{eq:orbit}), we obtain $Im \psi\subset A\cdot x$.  The last statement of the theorem follows from the fact that the image of a morphism is a Zariski constructible set [Hum].  The image of an irreducible variety under a morphism is also irreducible. 
\end{proof}

 
\begin{rem}
As mentioned at the beginning of the section, the proof of Proposition \ref{prop:constr} presented here differs from the one in [KW1, Theorem 3.7].   They prove a stronger result for $\fgn=\fgl(n)$ that does not require that $x$ is strongly regular.  However, we will only need the strongly regular case.
\end{rem}

We now turn our attention to the study of the action of the group $A$ on $\sreg$.  One way to approach this is to study the moment map for the group $A$.  The connected components of regular level sets of this map are orbits of strongly regular elements under the action of $A$.  The next section discusses the properties of this map.


\subsection{The moment map for the $A$-action}\label{ss:moment}
	
	We now study the map $\Phi: \fgn\to \Co^{d+ r_{n}}$ defined by
	\begin{equation}\label{eq:map}
	\Phi(x)=(f_{1,1}(x_{1}), f_{2,1}(x_{2}),\cdots, f_{n, r_{n}}(x)).
	\end{equation}
For $c\in\Co^{d+r_{n}}$ denote $\Phi^{-1}(c)=(\mathfrak{g}_{n})_{c}$.   It is a basic fact from Poisson geometry that the action the group $A$ preserves the fibres $\fibre$.  Let us denote the open subset of strongly regular elements in the fibre $(\mathfrak{g}_{n})_{c}\cap \mathfrak{g}_{n}^{sreg}$ by $\sfibre$.   One of the deep results in [KW1] is that for $\fg_{n}=\fgl(n)$, $\sfibre$ is non-empty for any $c\in\Co^{d+r_{n}}=\Co^{\frac{n(n+1)}{2}}$ (see Theorem 2.3).  This is not necessarily the case for $\fg_{n}=\fso(n)$.  However, we will consider a special class of $c\in\Co^{d+r_{n}}$ for which the statement is true in section \ref{s:generics}.  Given the assumption that $\sfibre\neq\emptyset$, the results we state in the rest of this section carry over to the orthogonal case.  

Let $x\in\sfibre$.  The connected components of $\sfibre$ are $A$-orbits in $\orbx^{sreg}=\orbx\cap\fgn^{sreg}$ and hence are Lagrangian submanifolds of $\orbx^{sreg}$ by Proposition \ref{prop:lag}.  The fibre $\sfibre$ also has the property that the connected components in the Euclidean topology and irreducible components in the Zariski topology coincide so that there are only finitely many orbits of the group $A$ in $\sfibre$. 



\begin{thm}\label{thm:cpts} 
Let $c\in\mathbb{C}^{d+r_{n}}$, with $c=(c_{1,1},\cdots, c_{i,j}, \cdots, c_{n, r_{n}})$, $c_{i,j}\in\Co$.  Let $\sfibre=\bigcup_{i=1}^{N(c)} (\fg_{n}^{sreg})_{c,i}$ be the irreducible component decomposition of the variety $\sfibre$.  Then $\sfibre$ is a smooth variety of pure dimension $d$.  Moreover, the irreducible components $ (\fg_{n}^{sreg})_{c,i}$ are precisely the $A$-orbits in $\sfibre$.  Hence for $x\in\sreg$, $A\cdot x$ is an irreducible, non-singular variety of dimension $d$.  
\end{thm}

\begin{proof}
In this proof overline denotes Zariski closure, unless otherwise stated.  The statement that $\sfibre$ is a smooth variety of pure dimension $d$ follows directly from Theorem 4 in [M, pg 172].  In the notation of that reference take $X=\fgn$, $Y=\overline{\sfibre}$, let $U$ run through all sets in a finite open, affine cover of $\fgn^{sreg}$, and take the functions $f_{k}$ to be $f_{i,j}(x)-c_{i,j}$ for $1\leq i\leq n,\, 1\leq j\leq r_{i}$.


 We now show that each irreducible component $(\fg_{n}^{sreg})_{c,i}$ is an $A$-orbit.  Let $x\in(\fg_{n}^{sreg})_{c,i}$ and consider the $A$-orbit through $x$, $A\cdot x$.  By Proposition \ref{prop:constr}, $A\cdot x$ is irreducible, which implies
\begin{equation}\label{eq:contain} 
\overline{A\cdot x}\subseteq\overline{ (\fg_{n}^{sreg})_{c,i}}.
\end{equation}
 Let $k=\dim \overline{A\cdot x}$.  Then by (\ref{eq:contain}) $k\leq d$.  We now show $k=d$.  By Proposition \ref{prop:constr}, $A\cdot x$ is a constructible subset of $\fgn$, so there exists a subset $U\subset A\cdot x$ which is open in $\overline{A\cdot x}$.  Let $W$ be the set of smooth points of $U$.  $W$ is then open in $U$, and therefore $\dim W=k$.  Since $W$ is a smooth subvariety of $\fgn$, it is an analytic submanifold of $\fgn$.  $W$ is then an open submanifold of the $d$-dimensional manifold $A\cdot x$.  Hence, $k=d$.  We thus have equality in (\ref{eq:contain})
\begin{equation}\label{eq:equality}
\overline{A\cdot x}=\overline{ (\fg_{n}^{sreg})_{c,i}}.
\end{equation}
Now, we observe
\begin{equation}\label{eq:closinter}
\overline{ (\fg_{n}^{sreg})_{c,i}}\cap\fgn^{sreg}=(\fg_{n}^{sreg})_{c,i}.
\end{equation}
Since $A\cdot x$ is constructible, its Zariski closure is the same as its closure in the Euclidean topology on $\fgn$ (see [M, pg 60]).  Equations (\ref{eq:equality}) and (\ref{eq:closinter}) then imply that $(\fg_{n}^{sreg})_{c,i}$ is $A$-invariant, and therefore $A\cdot x\subset(\fg_{n}^{sreg})_{c,i}$.  Now, we suppose that $A\cdot x\neq (\fg_{n}^{sreg})_{c,i}$.  Let $y\in (\fg_{n}^{sreg})_{c,i}-A\cdot x$.  The same argument applied to $y$ implies that $\overline{A\cdot y}=\overline{(\fg_{n}^{sreg})_{c,i}}$.  But $A\cdot y$ contains a Zariski open subset of $\overline{(\fg_{n}^{sreg})_{c,i}}$, and hence $A\cdot x\cap A\cdot y\neq \emptyset$, by the irreducibility of $\overline{(\fg_{n}^{sreg})_{c,i}}$.  We have obtained a contradiction and therefore $A\cdot x=(\fg_{n}^{sreg})_{c,i}$.  Repeating this argument for each irreducible component $(\fg_{n}^{sreg})_{c,i}$, $1\leq j\leq N(c)$, we obtain the desired result.

\end{proof}


\section{The Action of the group $A$ on Generic Matrices}\label{s:generics}

For $x\in\fg_{i}$ let $\sigma(x)$ denote the spectrum of $x$, where $x$ is viewed as an element of $\fg_{i}$.  We consider the following set of regular semisimple elements of $\fgn$.  
\begin{equation}
\fgno=\{x\in\fgn |\;x_{i}\text { is regular semisimple}, \, \sigma(x_{i-1})\cap \sigma (x_{i})=\emptyset, \, 2\leq i\leq n\}.
\end{equation}
In the case of $\fgl(n)$, $\fgno$ consists of matrices each of whose cutoffs are diagonalizable with distinct eigenvalues and no two adjacent cutoffs share any eigenvalues.  Here is an example of such a matrix.
\begin{exam}\label{ex:omega}
Consider the matrix in $\fgl(3)$ 
$$
X=\left[\begin{array}{ccc}
0 & 20 & 28\\
1& 1 &-14\\
0 & 1 &2\end{array}\right]_{\mbox{\large .}}
$$
 One can compute that $X$ has eigenvalues $\sigma(X)=\{-2, 2, 3\}$ so that $X$ is regular semisimple and that $\sigma(X_{2})=\{5,-4\}$.  Clearly $\sigma(X_{1})=\{0\}$.  Thus $X\in \fgl(3)_{\Omega}$.
\end{exam}

Let $d=\frac{1}{2}\dim\orbx,$ $\orbx$ a regular adjoint orbit in $\fgn$.  By (\ref{eq:sumranks}), $d=\sum_{i=1}^{n-1} r_{i}$.  Given $c\in\Co^{d+r_{n}}$, we write $c=(c_{1}, \cdots, c_{i},\cdots, c_{n})\in\Co^{r_{1}}\times\cdots\times \Co^{r_{i}}\times\cdots\times \Co^{r_{n}}$ with $c_{i}\in\Co^{r_{i}}$.  We identify $\Co^{r_{i}}\simeq \fh_{i}/W_{i}$ using the map in (\ref{eq:worb}).  We define a subset $\Omega_{n}\subset\Co^{d+r_{n}}$ as the set of $c$ such that $c_{i}$ and $c_{i+1}$ are regular orbits whose elements share no eigenvalues in common.  $\Omega_{n}$ is Zariski open in $\fgn$ by Remark 2.16 in [KW1].  With this definition, it is easy to note that $\fgno=\bigcup_{c\in\Omega_{n}} \fibre$.  This follows from the fact that if $x\in\fg_{i}$ takes the same value on the fundamental $\Ad$-invariants as a regular semisimple element, then it is conjugate to that element.  We are now ready to state the theorem concerning the orbit structure of the group $A$ on $\fgno$.  In the case of $\fgl(n)$ this theorem is due to Kostant and Wallach (see [KW1] Theorems 3.23, 3.28), and in the case of $\fso(n)$, it is due to [Col].



\begin{thm} \label{thm:generics}
The elements of $\fgno$ are strongly regular and therefore $\fg_{n}^{sreg}$ is non-empty.  If $c\in\Omega_{n}$, then $(\mathfrak{g}_{n})_{c}=\sfibre$ is precisely one $A$-orbit.  Moreover, $(\mathfrak{g}_{n})_{c}$ is a homogeneous space for a free, algberaic action of the torus $(\mathbb{C}^{\times})^{d}$. 
\end{thm}

\begin{rem}
In Remark \ref{r:sreg}, we noted that $\fgl(n)^{sreg}\neq\emptyset$.  This can be shown without use of Theorem \ref{thm:generics} (see Theorem 2.3. in [KW1]).  At this point, Theorem \ref{thm:generics} is our only way of producing strongly regular elements in $\fso(n)$.  
\end{rem}

We give a complete proof of this statement for $\fso(n)$ in section \ref{s:orthogen}.  We sketch the proof for $\fgl(n)$ in the next section, concentrating on the case $n=3$ for the illustration of the main ideas.  We conclude this section with some corollaries of Theorem \ref{thm:generics}. The fact that $\fgn^{sreg}$ is non-empty implies the following. 

\begin{cor}\label{c:algind}
Let $f_{i,j}\in P(\fg_{i})^{G_{i}}$ for $1\leq j\leq r_{i}$ generate the ring $P(\fg_{i})^{G_{i}}$.  Then the functions $\{f_{i,j}| 1\leq i\leq n, \, 1\leq j\leq r_{i}\}$ are algebraically independent over $\Co$. 
\end{cor}

\begin{cor}
The classical analogue of the Gelfand-Zeitlin algebra $J(\fgn)\subset P(\fgn)$ defined in equation (\ref{eq:invaralg}) is isomorphic as an associative algebra to the Gelfand-Zeitlin subalgebra $GZ(\fgn)$ of $U(\fgn)$.
\end{cor} 
\begin{proof}
The corollary follows from Remark \ref{r:GZ} and Corollary \ref{c:algind}. 

\end{proof}


\begin{cor}
Let $x\in\fgno$ and let $\orbx$ be its adjoint orbit.   Let $\orbx^{sreg}=\orbx\cap\fgn^{sreg}$.  The $A$-orbits in $\orbx^{sreg}$ are Lagrangian submanifolds of $\orbx$ and form the leaves of a polarization of $\orbx^{sreg}$.  
\end{cor}
\begin{proof}
The corollary follows directly from Proposition \ref{prop:lag}.
\end{proof}

We now obtain the complete integrability of the Gelfand-Zeitlin system on certain regular adjoints orbits in $\fgn$.
\begin{cor}\label{c:compint}
Let $x\in\fgno$ and let $\orbx$ be its adjoint orbit.  Let $q_{i,j}=f_{i,j}|_{\orbx}$, $1\leq i\leq n-1$, $1\leq j\leq r_{i}$.  The functions $\{q_{i,j} |  1\leq i\leq n-1,\, 1\leq j\leq r_{i}\}$ form a completely integrable system on $\orbx$.  
\end{cor}
\begin{proof}
The corollary follows from the definition of strong regularity in Theorem-Definition \ref{d:sreg} and its proof.  
\end{proof}

\begin{rem}
If $\fgn=\fgl(n)$, then Remark \ref{r:sreg} states that $\orbx^{sreg}$ is non-empty for any regular $x\in\fgl(n)$.  Thus, the Gelfand-Zeitlin system in completely integrable on any regular adjoint orbit in $\fgl(n)$.  
\end{rem}


\subsection{ The generic general linear case }\label{s:glngenerics}

Let us first consider Theorem \ref{thm:generics} in the case of $\fgl(3)$.  The idea behind the proof is to reparameterize the action of $A$ given by the composition of the flows in (\ref{eq:flows}) by a simpler action of $(\Co^{\times})^{3}$ which allows us to count $A$-orbits in the fibre $\fgl(3)_{c}^{sreg}$.  Equation (\ref{eq:flows}) implies the $A$-orbit of $x\in\fgl(3)$ is
\begin{equation}\label{eq:aact}
\Ad\left(\left[\begin{array}{ccc}
z_{1} & & \\
 & 1 & \\
 & & 1\end{array}\right] \left[\begin{array}{ccc}
z_{2} & & \\
 & z_{2} & \\ 
 & & 1\end{array}\right]\left[\begin{array}{ccc}
\multicolumn{2}{c}{\exp(tx_{2})} &  \\
 &  & \\ 
 & & 1 \end{array}\right]\right)\cdot x,
 \end{equation}
 where $z_{1},\, z_{2}\in\Co^{\times}$ and $t\in\Co$.  The difficulty with analyzing this action is that even for $x\in\fgl(3)_{\Omega}$, $\exp (tx_{2})$ can be complicated (see Example \ref{ex:omega}).  (For larger values of $n$ the flows of the vector fields $\xi_{f_{i,j}}$ for $j>1$ are much more complicated to handle (see (\ref{eq:flows})).)  If we let $Z_{GL(i)}(x_{i})\subset GL(i)$ be the centralizer of $x_{i}$ in $GL(i)$, we observe from (\ref{eq:aact}) that the action of $A$ appears to push down to an action of $Z_{GL(1)}(x_{1})\times Z_{GL(2)}(x_{2})$.  Since $x\in\fgl(3)_{\Omega}$, we should expect the orbits of $A$ to be given by orbits of an action of $(\Co^{\times})^{3}$. 
 
The construction of the action of $(\Co^{\times})^{3}$ begins by considering the following elementary question in linear algebra, which we state in a more general setting.   Suppose that we are given an $(i+1)\times (i+1)$ matrix of the following form
\begin{equation}\label{eq:easymatrix}
\left[\begin{array}{cc}
\begin{array}{cccc}
\mu_{1}& 0 & \cdots & 0\\
0& \mu_{2}&\ddots &\vdots\\
\vdots&\;& \ddots & 0\\
0& \cdots &\cdots &\mu_{i}\\
\end{array} & \begin{array}{c}
y_{1}\\
\vdots\\
\vdots\\
y_{i}\end{array}\\
\begin{array}{cccc}
z_{1}&\cdots&\cdots&z_{i}
\end{array} & w
\end{array}\right]
\end{equation}
with $\mu_{j}\neq\mu_{k}$.  We want to determine the values of the $z_{i},\, y_{i},\,\text{and } w$ that force the matrix in (\ref{eq:easymatrix}) to have characteristic polynomial $f(t)=\prod_{j=1}^{i+1} (\lambda_{j}-t)$ with $\lambda_{k}\neq\lambda_{j}$ and $\lambda_{j}\neq\mu_{k}$.  These values can be found by equating the characteristic polynomial of the matrix in (\ref{eq:easymatrix}) evaluated at $\mu_{j}$ to $f(\mu_{j})$ for $1\leq j\leq i$ and solving the resulting system of equations.  Performing this calculation, we find the matrix in (\ref{eq:easymatrix}) has characteristic polynomial $f(t)$ if and only if it is of the form 
\begin{equation}\label{eq:solmatrix}
\left[\begin{array}{cc}
\begin{array}{cccc}
\mu_{1}& 0 & \cdots & 0\\
0& \mu_{2}&\ddots &\vdots\\
\vdots&\;& \ddots & 0\\
0& \cdots &\cdots &\mu_{i}\\
\end{array} & \begin{array}{c}
-z_{1}^{-1}\zeta_{1}\\
\vdots\\
\vdots\\
-z_{i}^{-1}\zeta_{i}\end{array}\\
\begin{array}{cccc}
z_{1}&\cdots&\cdots&z_{i}
\end{array} & w
\end{array}\right]_{\mbox{\large ,}}
\end{equation}
where $z_{j}\in\mathbb{C}^{\times}$, $\zeta_{j}\neq 0$ and depends only on the eigenvalues $\mu_{l}$ and $\lambda_{k}$ for $1\leq l\leq i$, $1\leq k\leq i+1$, and $w=\sum_{j=1}^{i+1}\lambda_{j}-\sum_{k=1}^{i}\mu_{i}$.  To avoid ambiguity, it is necessary to fix an ordering of the eigenvalues of the $i\times i$ cutoff of the matrix in (\ref{eq:solmatrix}).  To do this, we introduce a lexicographical ordering on $\C$ defined as follows. Let $z_{1},\, z_{2}\in \C$.  We say that $z_{1}>z_{2}$ if and only if $Re z_{1}> Re z_{2}$ or if $Re z_{1}=Re z_{2}$ then $Im z_{1}>Im z_{2}$.  
\begin{dfn}\label{d:sol}
Let $c\in\Omega_{n}$, $c=(c_{1},\cdots, c_{i}, c_{i+1}, \cdots, c_{n})$ with $c_{i}\in \Co^{r_{i}}=\Co^{i}$.  Suppose that the regular semisimple orbit represented by $c_{i}$ consists of matrices with characteristic polynomial $\prod_{j=1}^{i} (\mu_{j}-t)$ and suppose that $\mu_{1}>\mu_{2}>\cdots >\mu_{i}$ in the lexicographical order on $\Co$.  We define $\sol$ as the elements of the form (\ref{eq:solmatrix}).  We refer to $\sol$ as the (generic) solution variety at level $i$.  
\end{dfn}

\begin{rem}
By (\ref{eq:solmatrix}) $\sol$ is isomorphic to $(\Co^{\times})^{i}$ as an algebraic variety.  We will identify $\sol$ with $(\Co^{\times})^{i}$ for the remainder of this section.   
\end{rem}



We now return to the case $n=3$. Let $x\in\fgl(3)_{c}$ with $c\in\Omega_{3}$.  Suppose that $c\in\Omega_{3}$ is such that $\sigma(x_{1})=\{a\}$, and $\sigma(x_{2})=\{\mu_{1},\mu_{2}\}$ with $\mu_{1}>\mu_{2}$ in lexicographical order.  Then $x_{2}\in\Xi^{1}_{c_{1},c_{2}}$, so that
\begin{equation}\label{eq:littlematrix}
x_{2}=\left[\begin{array}{cc} a & -z_{1}^{-1}\zeta_{1}\\
z_{1} & z \end{array}\right]_{\mbox{\large ,}}
\end{equation}
with $z_{1}\in\Co^{\times}$ and $\zeta_{1}\neq 0$ and independent of $z_{1}$.  There exists a morphism $\Co^{\times}\to GL(2)$, $z_{1}\to \gamma_{1}(z_{1})$ such that 
 \begin{equation}\label{eq:diag}
 \Ad(\gamma(z_{1})) x=\left[\begin{array}{ccc} \mu_{1}& 0 & -z_{2}^{-1}\zeta_{2}\\
0 & \mu_{2} &-z_{3}^{-1}\zeta_{3}\\
z_{2}& z_{3} & w\end{array}\right ]
\end{equation}
is in $\Xi^{2}_{c_{2}, c_{3}}$.  (This can be checked by explicit computation.)
%

Using equation (\ref{eq:diag}), we can define an isomorphism of affine varieties $\Gamma_{3}^{c}: (\Co^{\times})^{3}\to \fgl(3)_{c}$,
  $$
 \Gamma_{3}^{c} (z_{1},z_{2},z_{3})=  \Ad(\gamma(z_{1})^{-1})  \left[\begin{array}{ccc} \mu_{1}& 0 & -z_{2}^{-1}\zeta_{2}\\
0 & \mu_{2} &-z_{3}^{-1}\zeta_{3}\\
z_{2}& z_{3} & w\end{array}\right ]_{\mbox{\large .}}
$$
 The map $\Gamma_{3}^{c}$ starts with an element of the solution variety at level $2$ and then conjugates the $2\times 2$ cutoff of this element into $z_{1}\in\Xi^{1}_{c_{1}, c_{2}}$.  
Using the isomorphism, $\Gamma_{3}^{c}$, we can define a free algebraic action of $(\Co^{\times})^{3}$ on $\fgl(3)_{c}$.  One can write down this action explicitly as follows.  Suppose that $\Gamma_{3}^{c}(z_{1}, z_{2}, z_{3})=x$, then for $(z_{1}^{\prime}, z_{2}^{\prime}, z_{3}^{\prime})\in (\Co^{\times})^{3}$  
\begin{equation}\label{eq:cact}
(z_{1}^{\prime}, z_{2}^{\prime}, z_{3}^{\prime}) \cdot x= \Ad\left(\left[\begin{array}{ccc}
z_{1}^{\prime} & & \\
 & 1 & \\
 & & 1\end{array}\right]\gamma(z_{1})^{-1} \left[\begin{array}{ccc}
z_{2}^{\prime} & & \\
 & z_{3}^{\prime} & \\
 & & 1\end{array}\right] \gamma (z_{1})\right)\cdot x.
\end{equation}

Note that this action is conjugation by the centralizers of the $i\times i$ cutoffs of $x$ starting with the $2\times 2 $ cutoff.   The difference with (\ref{eq:aact}) is that we now diagonalize the $2\times 2$ cutoff before performing the conjugation, which makes the action much easier to understand.  Thus, it is reasonable to believe that the orbits of the action in (\ref{eq:aact}) and the action in (\ref{eq:cact}) coincide.  To show this precisely, one must first show that $\fgl(3)_{c}=\fgl(3)^{sreg}_{c}$.  Then the fact that $\fgl(3)_{c}$ is one $A$-orbit follows immediately from the irreducibility of $\fgl(3)_{c}$ and Theorem \ref{thm:cpts}.  One way of showing that $\fgl(3)_{c}=\fgl(3)^{sreg}_{c}$ is to show that the tangent space $T_{x}(\fgl(3)_{c})=V_{x}$, with $V_{x}$ as in (\ref{eq:easydist}).  Since $\dim(T_{x}(\fgl(3)_{c}))=3$, $\fgl(3)_{c}=\fgl(3)^{sreg}_{c}$ from (\ref{eq:sreg}).  We will compute the tangent space $T_{x}((\fgn)_{c})$, $c\in\Omega_{n}$ for $\fgn=\fso(n)$ in subsection \ref{s:proof}.  The computation in the case of $\fgn=\fgl(n)$ is analogous.  (In the case of $\fgl(n)$, one can also obtain $\fgl(3)_{c}=\fgl(3)^{sreg}_{c}$ by appealing directly to Theorem 2.17 in [KW1].)  


  
The general case proceeds similarly; we briefly summarize it here.  Let $(\mathbf{z_{1}}, \mathbf{z_{2}},\cdots, \mathbf{z_{n}})\in\Co^{\times}\times\cdots\times(\Co^{\times})^{i}\times\cdots\times(\Co^{\times})^{n-1}=(\Co^{\times})^{{n\choose 2}}$ with $\mathbf{z_{i}}=(z_{i,1},\cdots ,z_{i,i})\in(\Co^{\times})^{i}\simeq \sol$.  One can write down a matrix $\gamma_{i,i+1}(\mathbf{z_{i}})$ which diagonalizes $\mathbf{z_{i}}$ and depends regularly on $\mathbf{z_{i}}$.  Now, we can define a bijective morphism as in the case of $n=3$ by

\begin{equation}\label{eq:gammagen}
\Gamma_{n}^{c}(\mathbf{z_{1}},\mathbf{z_{2}},\cdots, \mathbf{z_{n}})=\Ad(\gamma_{1,2}(\mathbf{z_{1}})^{-1}\gamma_{2,3}(\mathbf{z_{2}})^{-1}
 \cdots \gamma_{n-2,n-1}(\mathbf{z_{n-2}})^{-1}) \cdot(\mathbf{z_{n-1}}).
\end{equation}
 One can show by explicit computation that $\Gamma_{n}^{c}:(\Co^{\times})^{{n\choose 2}}\to \glfibre$ is an isomorphism of affine varieties.

 
 \begin{rem}
 The proof of Theorem \ref{thm:generics} in the case of $\fgn=\fgl(n)$ in [KW1] (see Theorems 3.23 and 3.28) is different than the one outlined here.  This technique goes back to some preliminary work of Kostant and Wallach.  It is emphasized here, since it generalizes to describe less generic orbits of the group $A$ (see [Col1]), as well as the generic orthogonal case.  
 \end{rem}
 \subsection{ The generic orthogonal case}\label{s:orthogen}
In this section, we prove Theorem \ref{thm:generics} for $\fso(n)$.  We prove the theorem by constructing an algebraic isomorphism $\Gamma_{n}^{c}: (\Co^{\times})^{d}\to \fso(n)_{c}$ for $c\in\Omega_{n}$.  As in the case of $\fgl(n)$, we start by considering a problem in linear algebra.  Let $\fh_{2l}=\bigoplus_{i=1}^{l}\fso(2)$ be the standard Cartan subalgebra of block diagonal matrices in $\fso(2l)$.  For $\fso(2l+1)$, let $\fh_{2l+1}=\bigoplus_{i=1}^{l}\fso(2)\oplus \{0\}$, where $\{0\}$ is the $1\times 1$ $0$-matrix, denote the standard Cartan subalgebra.  Let $\fh^{reg}$ denote the regular elements in the Cartan $\fh$.  Suppose we are given an element in $\fso(i+1)$ of the form	\begin{equation}\label{eq:orthomatrix}
	  \left[\begin{array}{cc}
 h & \underline{z}\\
 -\underline{z}^{t} & 0\end{array}\right ]_{\mbox{\large ,}}
\end{equation}
where $h\in\fh_{i}^{reg}\subset\mathfrak{so}(i)$ is in the $W_{i}$ orbit determined by $c_{i}\in\fh_{i}^{reg}/W_{i}$ and $\underline{z}\in\Co^{i}$ is a column vector.  Suppose we are given $c_{i+1}\in\fh_{i+1}^{reg}/W_{i+1}$, a regular semisimple orbit whose elements have no eigenvalues in common with those of $c_{i}$.  The problem is to determine the value of $\underline{z}$ that forces the matrix in (\ref{eq:orthomatrix}) to lie in the orbit $c_{i+1}$.  

To avoid ambiguity, we need to choose a fundamental domain $\mathcal{D}_{i}$ for the action of the Weyl group $W_{i}$ on $\fh_{i}^{reg}$.  Although we will not make explicit use of it, we give an example of such a fundamental domain for completeness.  Let $\Phi(\fso(i), \fh_{i})=\Phi$ be a system of roots relative to the Cartan subalgebra $\fh_{i}$, and let $\Phi^{+}$ denote a choice of positive roots.  Given $z\in\mathfrak{h}_{i}$, we write $z=Re z+\imath \,Im z$ with $Re z \in(\fh_{i})_{\mathbb{R}}$ and $Im z\in (\fh_{i})_{\mathbb{R}}$.  It is a standard result that a fundamental domain for the action of $W_{i}$ on $\mathfrak{h}_{i}^{reg}$ is (see [CM])
\begin{equation}\label{eq:fundom}
\mathcal{D}_{i}=\{z\in \fh_{i}^{reg}| \alpha(Rez)\geq 0, \text{ if } \alpha(Re z)=0,\, \alpha(Im z)>0\text{ for all }\alpha\in\Phi^{+}\}.
\end{equation}

\begin{dfn}\label{d:orthosol}
Let $c_{i}\in\fh_{i}^{reg}/ W_{i}$ and $c_{i+1}\in \fh_{i+1}^{reg}/ W_{i+1}$ be regular semisimple adjoint orbits in $\fso(i)$ and $\fso(i+1)$ respectively whose elements contain no eigenvalues in common.  We define $\sol$ to be the set of matrices in $\fso(i+1)$ of the form (\ref{eq:orthomatrix}) which are in the regular semisimple orbit $c_{i+1}$ with $h\in \mathcal{D}_{i}$.  We refer to $\sol$ as the (generic) orthogonal solution variety at level $i$.    
\end{dfn}  

There are now two types different types of solution varieties depending on the type of $\fso(i+1)$.  The more subtle case is when $\fso(i+1)=\fso(2l+2)$ is of type $D_{l+1}$.  We will carefully study the geometry of $\sol$ in this case in subsection \ref{s:dlsol}.  In subsection \ref{s:blsol}, we will sketch the analogous results for when $\fso(i+1)=\fso(2l+1)$ is of type $B_{l}$.  This last case is dealt with similarly to the one in \ref{s:dlsol}.  In subsection \ref{s:proof}, we use the results about the varieties $\sol$ to construct the isomorphism $\Gamma_{n}^{c}$ and prove Theorem \ref{thm:generics}.  

\subsubsection{Solution varieties in type $D_{l+1}$.} \label{s:dlsol}
Let $c_{2l+1}\in\fh_{2l+1}/ W_{2l+1}$ and $c_{2l+2}\in\fh_{2l+2}/ W_{2l+2}$ be regular semisimple orbits whose elements have no eigenvalues in common.  We now show that $\Xi^{2l+1}_{c_{2l+1}, c_{2l+2}}$ is non-empty.  Let $h\in\mathcal{D}_{2l+1}\cap c_{2l+1}$ with $\mathcal{D}_{2l+1}$ as in (\ref{eq:fundom}).  Suppose $h=\oplus_{i=1}^{l} a_{i} J+\{0\}$ with $J=\left[ \begin{array}{cc} 
0 & 1\\
-1 & 0\end{array}\right]
$ and $\{0\}$ denoting the $1\times 1$ $0$-matrix, with $a_{i}\neq (+/-) a_{j}$ for $j\neq i$ and $a_{i}\neq 0$ for all $i$. Suppose that the orbit $c_{2l+2}$ consists of elements with characteristic polynomial 
\begin{equation}\label{eq:Dchar}
\prod_{i=1}^{l+1} (t^{2}+b_{i}^{2})
\end{equation} 
and Pfaffian either 
\begin{equation}\label{eq:Pfaffs}
\prod_{i=1}^{l+1} b_{i} \text{ or } -\prod_{i=1}^{l+1} b_{i}.
\end{equation}

We consider elements of $\fso(2l+2)$ of the form 
\begin{equation}\label{eq:bd}
X=\left[\begin{array}{ccccccccc}
   0 & a_{1}&  & & & & &  &z_{11}\\
	-a_{1}& 0 &   &  & & & & & z_{12}\\
	 &  &0 & a_{2} & &  & & &z_{21}   \\
	&  & -a_{2}& 0 & &  & & &z_{22}\\
	& &  &  & \ddots &   & & &\vdots   \\
	&  &  & &  &  0 &a_{l}& &z_{l1} \\
	  &  & &  &  & -a_{l} & 0 &  &z_{l2}\\
	  & & & & & & & 0 & z_{l+1}\\
	 -z_{11}& -z_{12}& -z_{21}& -z_{22}& \cdots & -z_{l1} &-z_{l2}& -z_{l+1} & 0\end{array}\right]_{\mbox{\large ,}}
	  \end{equation}  
as in (\ref{eq:orthomatrix}).  We need to show that the coordinates $z_{i,j}$ and $z_{l+1}$ can be chosen to give $X$ characteristic polynomial (\ref{eq:Dchar}) and either choice of Pfaffian in (\ref{eq:Pfaffs}).  Then $X\in c_{2l+2}$, since it takes the same values on the fundamental $\Ad$-invariants as an element of $c_{2l+2}$.  We begin by computing the characteristic polynomial of $X$.   

\begin{lem}\label{l:Dchar}
The characteristic polynomial of the matrix in (\ref{eq:bd}) is
\begin{equation}\label{eq:dpoly}
\sum_{i=1}^{l}\left((z_{i1}^{2}+z_{i2}^{2})\,t^{2}\,\left[\prod_{j=1,\,j\neq i}^{l}(t^{2}+a_{j}^{2})\right]\right)+z_{l+1}^{2}\prod_{i=1}^{l}(t^{2}+a_{i}^{2})+t^{2}\prod_{i=1}^{l}(t^{2}+a_{i}^{2}).
\end{equation}
\end{lem}

\begin{proof}
 
We want to compute $\det (t-X)$.  We compute the determinant using the Schur complement determinant formula (see [HJ, pgs 21-22]).  In the notation of that reference $\alpha=\{1,\cdots, n-1\}$ and $\alpha^{\prime}=\{n\}$.  The Schur complement formula gives
\begin{equation}\label{eq:Schur}
\det(t-X)=\det(t-h)\,[t+\underline{z}^{T}\; (t-h)^{-1}\;\underline{z}],
\end{equation}
where $\underline{z}^{T}=[z_{11}, z_{12}, \cdots, z_{l1}, z_{l2}, z_{l+1}]$ is a row vector (see (\ref{eq:bd})).  We compute that $(t-h)^{-1}$ is the block diagonal matrix
\begin{equation}\label{eq:adjmatrix}
\bigoplus_{j=1}^{l}\, \frac{1}{ (t^{2}+a_{j}^{2})} \left[\begin{array}{cc} t & a_{j}\\
												                                                             -a_{j} & t \end{array}\right] \oplus \frac{1}{t}.
\end{equation}
We then compute that $t+\underline{z}^{T} (t-h)^{-1}\underline{z}$ is
\begin{equation}\label{eq:adj}
t+\sum_{j=1}^{l} \frac{t}{t^{2}+a_{j}^{2}}(z_{j1}^{2}+z_{j2}^{2})+ \frac{z_{l+1}^{2}}{t}.
\end{equation}
We see easily that $\det(t-h)=t\prod_{i=1}^{l} (t^{2}+a_{i}^{2})$.  Multiplying (\ref{eq:adj}) by $\det(t-h)$ yields (\ref{eq:dpoly}).  


\end{proof}

By Lemma \ref{l:Dchar} the matrix in (\ref{eq:bd}) has characteristic polynomial (\ref{eq:Dchar}) if and only if  

\begin{equation}\label{eq:cond1}
	  z_{i1}^{2}+z_{i2}^{2}=\frac{\prod_{j=1}^{l+1}(b_{j}^{2}-a_{i}^{2})}{-(a_{i}^{2})\prod_{j=1,\,j\neq i}^{l}(a_{j}^{2}-a_{i}^{2})} \text  { for } 1\leq i\leq l \text{ and }
	  \end{equation}

\begin{equation}\label{eq:cond3}
z_{l+1}^{2}=\frac{\prod_{i=1}^{l+1} b_{i}^{2}}{\prod_{i=1}^{l}a_{i}^{2}}.
\end{equation}
(These conditions are obtained by substituting the $2l+1$ distinct eigenvalues of $h$ into both (\ref{eq:dpoly}) and (\ref{eq:Dchar}) and equating the results.)  The right hand sides of (\ref{eq:cond1}) and (\ref{eq:cond3}) are defined precisely because of our assumption that $h$ is regular.   Moreover, the right hand sides of both (\ref{eq:cond1}) and (\ref{eq:cond3}) are non-zero because $h$ does not share any eigenvalues with elements of $c_{2l+2}$.  We write (\ref{eq:cond1}) as
\begin{equation}\label{eq:cond2}
(z_{i1}^{2}+z_{i2}^{2})=d_{i}
\end{equation}
with $d_{i}\neq 0$, $$d_{i}=\frac{\prod_{j=1}^{l+1}(b_{j}^{2}-a_{i}^{2})}{-(a_{i}^{2})\prod_{j=1,\,j\neq i}^{l}(a_{j}^{2}-a_{i}^{2})}.$$  Observe that $d_{i}$ only depends upon the values of the $a_{j}$ and the $b_{k}$.  
%
  We have some ambiguity in the choice of $z_{l+1}$.  It is exactly this ambiguity that allows us to prescribe the value of $z_{l+1}$ so that the Pfaffian of the matrix in (\ref{eq:bd}) can be made to be either expression in (\ref{eq:Pfaffs}).  Indeed, Lemma \ref{l:Dchar} gives that the determinant of $X$ is 
 $$
z_{l+1}^{2}\prod_{i=1}^{l}a_{i}^{2}.
$$
Therefore, the Pfaffian of $X$ is $(+/-)z_{l+1}\prod_{i=1}^{l}a_{i}$.   
Thus, we can choose the sign of $z_{l+1}$ in (\ref{eq:cond3}) so that the Pfaffian of the matrix (\ref{eq:bd}) is either $\prod_{i=1}^{l+1} b_{i}$ or $-\prod_{i=1}^{l+1} b_{i}$. Thus, the generic orthogonal solution variety $\soll$ is non-empty in this case. 



	 We now study the geometry of $\soll$ in more detail.    
	 We consider the equation (\ref{eq:cond2}) for $i=1,\cdots, l$.  We define a closed subvariety of $\Co^{2}$ 
$$ V_{i}=\{(z_{i1},\,z_{i2})\in\mathbb{C}^{2}| z_{i1}^{2}+z_{i2}^{2}=d_{i}\}.
$$ 
Let $SO(2)$ (thought of as $2\times 2$ complex orthogonal matrices) act on $\Co^{2}$ as linear transformations.  It is then easy to see that this action preserves $V_{i}\subset\Co^{2}$.  Moreover, there is a natural $SO(2)$-equivariant isomorphism $V_{i}\simeq SO(2)$, where we think of $SO(2)$ acting on itself by left translation.

Thus, $\soll$ is isomorphic to $SO(2)^{l}$.  The centralizer of $h$ in $SO(2l+1)$ acts on $\soll$ by conjugation.  Since $h$ is regular, $Z_{SO(2l+1)}(h)=SO(2)^{l}\times \{1\}$, where $\{1\}$ denotes the $1\times 1$ identity matrix.  The action of $Z_{SO(2l+1)}(h)$ on $\soll$ is the standard diagonal action of $SO(2)^{l}$ on the column vector $[z_{11}, z_{12},\cdots, z_{l1}, z_{l2}]^{T}$ (see (\ref{eq:bd})), and the dual action on the row vector\\ $-[z_{11}, z_{12},\cdots, z_{l1}, z_{l2}]$.  Under the identification $\soll\simeq SO(2)^{l}$ this action can be identified with the action of $SO(2)^{l}$ on itself by left translation.  Since this action of $Z_{SO(2l+1)}(h)$ is free, it follows that $Z_{SO(2l+1)}(h)$ acts simply transitively on $\soll$.  We have now proven the following theorem.  


\begin{thm}\label{thm:bdextension}
  Given any regular semisimple orbits $c_{2l+1}\in\fh_{2l+1}^{reg}/W_{2l+1}$ and $c_{2l+2}\in\fh_{2l+2}^{reg}/W_{2l+2}$ whose elements share no eigenvalues in common, the solution variety $\soll$ is non-empty and is a homogenuous space for a free, algebraic action of $Z_{SO(2l+1)}(h)= SO(2)^{l}$.  Moreover, any element in $\fso(2l+2)$ of the the form (\ref{eq:bd}) which is in the orbit $c_{2l+2}\in \fh_{2l+2}^{reg}/ W_{2l+2}$ is necessarily in $\soll$.  
	\end{thm}

\subsubsection{ Solution varieties in type $B_{l}$}\label{s:blsol} 
Let $c_{2l}\in\fh_{2l}^{reg}/ W_{2l}$ and $c_{2l+1}\in\fh_{2l+1}^{reg}/W_{2l+1}$ be regular semisimple orbits whose elements have no eigenvalues in common.
 Let $h\in \mathcal{D}_{2l}\cap c_{2l}\subset \fso(2l)$ be the block diagonal matrix $h=\oplus_{i=1}^{l} a_{i} J$ with $J\in\fso(2)$ as in the previous section.  We are forced to have $a_{i}\neq 0$ for any $i$ and that $a_{i}\neq (+/-) a_{j}$ for $i\neq j$.  For this choice of $h$ the matrix in (\ref{eq:orthomatrix}) can be written as
 \begin{equation}\label{eq:extension2}
Y=\left[\begin{array}{cccccccc}
   0 & a_{1}&  & & & & &z_{11}  \\
	-a_{1}& 0 &   &  & & & & z_{12}  \\
	 &  &0 & a_{2} & &  & &z_{21}   \\
	&  & -a_{2}& 0 & &  & & z_{22}\\
	& &  &  & \ddots &   & &\vdots   \\
	&  &  & &  &  0 &a_{l}&z_{l1}  \\
	  &  & &  &  & -a_{l} & 0 & z_{l2}\\
	 -z_{11}& -z_{12}& -z_{21}& -z_{22}& \cdots & -z_{l1} &-z_{l2}& 0\end{array}\right]_{\mbox{\large .}}
	  \end{equation}  

Suppose that the orbit $c_{2l+1}$ is the set of elements in $\fso(2l+1)$ with characteristic polynomial
\begin{equation}\label{eq:extend}
 t\,\prod_{i=1}^{l} (t^{2}+b_{i}^{2}),
\end{equation}
with $b_{i}\in\mathbb{C}$ satisfying the conditions that $b_{j}\neq (+/-) a_{k}$, $b_{i}\neq (+/-) b_{j}$, and $b_{i}\neq 0$.   We want to find values for the $z_{i,j}$ so that $Y$ has characteristic polynomial (\ref{eq:extend}).  The matrix $Y$ in (\ref{eq:extension2}) is then in the orbit $c_{2l+1}$.  
As in the previous case, we begin by computing the characteristic polynomial of $Y$.  The following lemma is an easy consequence of Lemma \ref{l:Dchar}.
\begin{lem}\label{l:Bchar}
 The characteristic polynomial of the matrix in (\ref{eq:extension2}) is 
\begin{equation}\label{eq:bpoly}
t\left[\prod_{i=1}^{l}(t^{2}+a_{i}^{2})+\sum_{i=1}^{l}(z_{i1}^{2}+z_{i2}^{2})\prod_{j=1,\,j\neq i}^{l} (t^{2}+a_{j}^{2})\right].
\end{equation}
\end{lem}
\begin{rem}\label{r:pfaff}
Notice that the polynomial in (\ref{eq:bpoly}) is invariant under sign changes $a_{j}\to -a_{j}$.  Thus, the result of Lemma \ref{l:Bchar} is independent of the Pfaffian of the matrix $h$. 
\end{rem}

The polynomial in (\ref{eq:bpoly}) is equal to the one in (\ref{eq:extend}) if and only if
\begin{equation}\label{eq:more}
(z_{i1}^{2}+z_{i2}^{2})=\frac{\prod_{j=1}^{l}(b_{j}^{2}-a_{i}^{2})}{\prod_{j=1,\,j\neq i}^{l}(a_{j}^{2}-a_{i}^{2})}, 
\end{equation}
for $1\leq i\leq l$.
The right hand side of (\ref{eq:more}) is defined precisely because $h$ is regular.  It is also non-zero because $h$ shares no eigenvalues with elements in the orbit $c_{2l+1}$.   If we let 
\begin{equation}\label{eq:thing}
d_{i}=\frac{\prod_{j=1}^{l}(b_{j}^{2}-a_{i}^{2})}{\prod_{j=1,\,j\neq i}^{l}(a_{j}^{2}-a_{i}^{2})}
\end{equation}
(\ref{eq:more}) becomes,
\begin{equation}\label{eq:onemore}
z_{i1}^{2}+z_{i2}^{2}=d_{i},
\end{equation}
with $d_{i}\neq 0$.  As in the previous section, $d_{i}$ depends only on the values of the $a_{j}$ and $b_{k}$.  Comparing (\ref{eq:onemore}) with (\ref{eq:cond2}) and using $Z_{SO(2l)}(h)=\prod_{i=1}^{l} SO(2)$, we can argue as we did in subsection \ref{s:dlsol} to obtain the following theorem.
\begin{thm}\label{thm:extensionDB}
  Given any regular semisimple orbits $c_{2l}\in\fh_{2l}^{reg}/W_{2l}$ and $c_{2l+1}\in\fh_{2l+1}^{reg}/W_{2l+1}$ whose elements have no eigenvalues in common, the solution variety $\Xi^{2l}_{c_{2l}, c_{2l+1}}$ is non-empty and is a homogenuous space for a free, algebraic action of $Z_{SO(2l)}(h)= SO(2)^{l}$.   Moreover, any element in $\fso(2l+1)$ of the the form (\ref{eq:extension2}) which is in the orbit $c_{2l+1}$ is necessarily in $\Xi^{2l}_{c_{2l}, c_{2l+1}}$.\end{thm}

\subsubsection{Proof of Theorem \ref{thm:generics} for $\fso(n)$. }\label{s:proof}

We can use our description of $\sol$ in Theorems \ref{thm:bdextension} and \ref{thm:extensionDB} to define a morphism 
$$
\Gamma_{n}^{c}: SO(2)^{d}\to\ofibre
$$
for $c\in\Omega_{n}$, as we did in the case of $\fg_{n}=\fgl(n)$.  However, there is one difficulty in this case that was not present in the generic case in $\fgl(n)$.  To define the map $\Gamma_{n}^{c}$, we must construct a morphism $\sol\to SO(i+1)$ which sends $z\to g_{i,i+1}(z)$, where $g_{i,i+1}(z)$ conjugates $z$ into the unique element in $c_{i+1}\cap \mathcal{D}_{i+1}$.  In the case of $\fgl(n)$, Wallach did this by explicit computation.  In this case, we have to be more indirect. 
Let $l=rank(\mathfrak{so}(i))=r_{i}$.  Let $h$ be the $i\times i$ cutoff of the matrices in (\ref{eq:bd}) and (\ref{eq:extension2}).  By Theorems \ref{thm:bdextension} and \ref{thm:extensionDB} we can identify $\sol\simeq SO(2)^{l}\simeq Z_{SO(i)}(h)$.  Under this identification the $SO(2)^{l}$ action on $\sol$ is identified with action of left translation of $SO(2)^{l}$ on itself.  To define $g_{i, i+1}(z)$, fix a choice of element $p_{i,i+1}\in SO(i+1)$ such that $\Ad(p_{i,i+1})Id_{i}\in\mathcal{D}_{i+1}\cap c_{i+1}$, where $Id_{i}\in SO(2)^{l}$ is the identity element.  

Let $z\in\Xi^{i}_{c_{i},\, c_{i+1}}$ be arbitrary.   Then 
\begin{equation}\label{eq:gmatrix}
g_{i,i+1}(z)=p_{i,i+1} z^{-1}
\end{equation}
 conjugates $z$ into $\mathcal{D}_{i+1}\cap c_{i+1}$ and the function 
\begin{equation}\label{eq:funmatrix}
z\to p_{i,i+1} z^{-1}
\end{equation}
is a morphism from $\Xi^{i}_{c_{i},\, c_{i+1}}$ to $SO(i+1)$.  

We can now define a morphism
	
	\begin{equation}\label{eq:orthodefn}
	\begin{array}{c}
\Gamma_{n}^{c}: SO(2)\times\cdots\times SO(2)^{r_{n-1}}\rightarrow \mathfrak{so}(n)_{c},\\
\\
\\
 \\
\Gamma_{n}^{c}(z_{2},\cdots, z_{n-1})=
\Ad(g_{2,3}(z_{2})^{-1} \cdots g_{n-2,n-1}(z_{n-2})^{-1}) z_{n-1}.\end{array}
\end{equation}
 \begin{rem}\label{r:image}
 To see that $\Gamma_{n}^{c}$ maps into $\ofibre$, note that for $Y\in Im\Gamma_{n}^{c}$, $Y_{i+1}\in \Ad(SO(i)) \cdot z_{i}$ for $3\leq i\leq n-1$.  Thus, $Y_{i+1}$ takes the same values on the fundamental $\Ad$-invariants for $\fso(i+1)$ as $z_{i}\in\sol$.  Note also that $Y_{3}=z_{2}\in\Xi^{2}_{c_{2}, c_{3}}$.  
 \end{rem}
The existence of this mapping gives us that $\fso(n)_{c}$ is non-empty.  As in the case of $\fgl(n)$, we have the following theorem concerning the morphism $\Gamma_{n}^{c}$.

\begin{thm}\label{thm:ofibre}
Let $c\in\mathbb{C}^{d+r_{n}}\in\Omega_{n}$ and let $d=\frac{1}{2}\dim\orbx$, $\orbx\subset\fso(n)$ a regular adjoint orbit.  Then the fibre $\mathfrak{so}(n)_{c}$ is non-empty.  The morphism $\Gamma_{n}^{c}$ is an isomorphism of affine varieties.  Therefore, $\mathfrak{so}(n)_{c}$ is a smooth, irreducible affine variety of dimension $d$.  
\end{thm}
\begin{proof}
We show the map $\Gamma_{n}^{c}$ is an isomorphism by explicitly constructing an inverse.  Specifically, we show that there exist morphisms $\psi_{i}: \ofibre\to SO(2)^{r_{i}}$ for $2\leq i\leq n-1$ so that the morphisim defined by 
\begin{equation}\label{eq:inverse}
\Psi=(\psi_{2},\cdots , \psi_{n-1}): \ofibre\to SO(2)\times\cdots\times SO(2)^{r_{n-1}}  
\end{equation}
is an inverse to $\Gamma_{n}^{c}$.  The morphisms $\psi_{i}$ are constructed inductively.  

Given $x\in\mathfrak{so}(n)_{c},\, x_{3}\in\Xi^{2}_{c_{2},\, c_{3}}$.   We then define $\psi_{2}(x)=x_{3}$.  By (\ref{eq:gmatrix}) the element $g_{2,3}(\psi_{2}(x))\in SO(3)$ which conjugates $x_{3}$ into $c_{3}\cap \mathcal{D}_{3}$ depends regularly on $\psi_{2}(x)$ and thus on $x$.  Thus, the map 
$$
x\to (\Ad(g_{2,3}(\psi_{2}(x)))\cdot x)_{4}\in \Xi^{3}_{c_{3},\, c_{4}}
$$
is a morphism.  This defines $\psi_{3}(x)=(\Ad(g_{2,3}(\psi_{2}(x)))\cdot x)_{4}$.  Now, suppose that we have defined morphisms $\psi_{2},\cdots, \psi_{m}$ for $2\leq m\leq j-1$, with $\psi_{m}:\ofibre\to SO(2)^{r_{m}}$.  Then the $(j+1)\times (j+1)$ cutoff of the matrix 
$$
 \Ad( g_{j-1,j}(\psi_{j-1}(x)))\cdots \Ad(g_{2,3}(\psi_{2}(x))) \cdot x.
$$
is in the solution variety at level $j$, $\Xi^{j}_{c_{j},\, c_{j+1}}$.  (The elements $g_{m, m+1}(\psi_{m}(x))$ are defined by (\ref{eq:gmatrix}).)  Thus, the map $\psi_{j}:\ofibre\to SO(2)^{r_{j}}$, 
\begin{equation}\label{eq:psij}
\psi_{j}(x)= [ \Ad( g_{j-1,j}(\psi_{j-1}(x)))\cdots \Ad(g_{2,3}(\psi_{2}(x))) \cdot x]_{j+1}
\end{equation}
is a morphism.  

Now, we need to see that the map $\Psi$ is an inverse to $\Gamma_{n}^{c}$.  We first show that $\Gamma_{n}^{c}(\psi_{2}(x),\cdots ,\psi_{n-1}(x))=x$.  Consider equation (\ref{eq:psij}) with $j=n-1$,  $$\psi_{n-1}(x)=  \Ad( g_{n-2, n-1}(\psi_{n-2}(x)))\cdots \Ad(g_{2,3}(\psi_{2}(x))) \cdot x.$$  Now, using the definition of $\Gamma_{n}^{c}$ in (\ref{eq:orthodefn}), it is clear that $\Gamma_{n}^{c}(\psi_{2}(x),\cdots ,\psi_{n-1}(x))=x$. 

Finally, we show $\Psi(\Gamma_{n}^{c}(z_{2},\cdots, z_{n-1}))=(z_{2},\cdots , z_{n-1})$.  Consider the element
$$
\Ad(g_{j,j+1}(z_{j})^{-1}
 \cdots g_{n-2,n-1}(z_{n-2})^{-1}) z_{n-1},
$$
for $2\leq j\leq n-2$.  The $(j+1)\times (j+1)$ cutoff of this element is equal to $z_{j}\in\Xi^{j}_{c_{j},\, c_{j+1}}$.  Using this fact with $j=2$, we see $\psi_{2}(x)=z_{2}$.  We work inductively, as we did in defining the map $\Psi$.  We assume $\psi_{2}(x)=z_{2}, \cdots , \psi_{m}(x)=z_{m}$ for $2\leq m\leq j-1$.  By the definition of $\psi_{j}$ in (\ref{eq:psij}), $\psi_{j}(x)=[\Ad(g_{j,j+1}(z_{j})^{-1}\cdots g_{n-2,n-1}(z_{n-2})^{-1}) z_{n-1}]_{j+1}=z_{j}$.  From which we obtain easily $\psi_{n-1}(x)=z_{n-1}$.  Thus, we obtain $\Psi\circ\Gamma_{n}^{c}=id$.  This completes the proof that $\Gamma_{n}^{c}$ is an algebraic isomorphism.

\end{proof}

\begin{rem} 
Note that as an algebraic group, $SO(2)^{d}\simeq(\Co^{\times})^{d}$.  Thus, Theorem \ref{thm:ofibre} is the orthogonal analogue of Theorem 3.23 in [KW1].
\end{rem}

\begin{rem}
In the case of $\fso(n)$, it is not automatic that the fibre $\mathfrak{so}_{c}(n)$ is non-empty for $c\in\mathbb{C}^{d+r_{n}}\in\Omega_{n}$.  In the case of $\fgl(n)$, we know that all fibres are non-empty for $c\in\mathbb{C}^{{n+1\choose 2}}$ because the moment map in (\ref{eq:map}) is surjective by Theorem 2.3 [KW1]. 
\end{rem}

We can use the isomorphism $\Gamma_{n}^{c}$ to define an algebraic action of $SO(2)^{d}=(\Co^{\times})^{d}$ on $\ofibre$. Let $g=(z_{2}^{\prime},\cdots, z_{n-1}^{\prime})\in SO(2)^{d}$.  For $x\in\ofibre$, suppose $\Psi(x)=(\Gamma_{n}^{c})^{-1}(x)=(z_{2},\cdots, z_{n-1})$, then 
\begin{equation}\label{eq:oaction}
g\cdot x=\Gamma_{n}^{c}(z_{2}^{\prime} z_{2},\cdots, z_{n-1}^{\prime} z_{n-1}
).
\end{equation}
 	The above action of $SO(2)^{d}$ on $\mathfrak{so}(n)_{c}$ is a simply transitive algebraic group action on $\mathfrak{so}(n)_{c}$.   

 Using the definition of the matrices $g_{j, j+1}(z_{j})$ in (\ref{eq:gmatrix}) it is easy to see
  \begin{equation}\label{eq:actcomp}
 g(z_{j}^{\prime} z_{j})=g(z_{j}) (z_{j}^{\prime})^{-1}.
 \end{equation}
 By (\ref{eq:actcomp}) $(z_{2}^{\prime},\cdots , z_{n-1}^{\prime})\cdot x$ can be written as
 
 \begin{equation}\label{eq:oeasyact}
 \Ad(z_{2}^{\prime}g_{2,3}(z_{2})^{-1}z_{3}^{\prime}g_{3,4}(z_{3})^{-1}
      \cdots g_{n-2,n-1}^{-1}(z_{n-2})^{-1} z_{n-1}^{\prime}
      g_{n-2,n-1}(z_{n-2}) \cdots g_{2,3}(z_{2})) \cdot x.\end{equation}
   Let $h_{i}\in c_{i}\cap\mathcal{D}_{i}$ for $2\leq i\leq n-1$ be the $i\times i$ cutoff of the matrix in (\ref{eq:orthomatrix}).  As in the case of $\fgl(n)$ (cf. (\ref{eq:cact})), this action of $(\Co^{\times})^{d}$ is a sequence of conjugations by $Z_{SO(i)}(h_{i})$ starting with $Z_{SO(n-1)}(h_{n-1})$.  From equation (\ref{eq:orbit}), we expect this action to have the same orbits as the action of $A$.  The only difference with this action is that we first conjugate the cutoff into the Cartan $\fh$ before performing the conjugation by its centralizer, which makes the action easier to understand.  We can now prove Theorem \ref{thm:generics} in the orthogonal case.  
        
\begin{proof}[Proof of Theorem \ref{thm:generics} for $\fg_{n}=\fso(n)$]
The first step is to see that $\ofibre^{sreg}=\ofibre$.  We will then obtain that $\ofibre$ is one $A$-orbit from Theorems \ref{thm:ofibre} and \ref{thm:cpts}.  To show $\ofibre^{sreg}=\ofibre$, we work analytically.  By Theorem \ref{thm:ofibre}, $\ofibre$ is a non-singular affine variety of dimension $d$.  Thus, it has the structure of an analytic submanifold of $\fso(n)$ of dimension $d$ and the map $\Gamma_{n}^{c}$ is a biholomorphism.  Thus, the action of $SO(2)^{d}$ in (\ref{eq:oeasyact}) is holomorphic.  To show $\ofibre^{sreg}=\ofibre$, we compute the tangent space $T_{x}(\ofibre)$ analytically and show that it is equal to $V_{x}$, $V_{x}$ as in (\ref{eq:odist}).  Since $\ofibre$ is one free orbit under the action of $SO(2)^{d}$ in (\ref{eq:oeasyact}), we can compute $T_{x}(\ofibre)$ by differentiating the action in (\ref{eq:oeasyact}) at the identity.  To differentiate this action of $SO(2)^{d}$, we use the following coordinates for $SO(2)$ in a neighbourhood of the identity.  Let   
   \begin{equation}\label{eq:sin}
  \underline{z}=\left[\begin{array}{cc}
   \cos(z) & \sin(z)\\
  -\sin(z)& \cos(z) \end{array}\right]_{\mbox{,}}
\end{equation}
for $z\in U\subset\mathbb{C}$, where $U$ is a neighbourhood of the origin in $\Co$.  

 
 We use as a basis for Lie($SO(2)$), $\frac{\partial}{\partial z}|_{z=0}$.  We recall that we are thinking of $SO(2)^{r_{i}}=Z_{SO(i)}(h_{i})$, with $h_{i}$ the $i\times i$ cutoff of the matrix in (\ref{eq:orthomatrix}) as in Theorems \ref{thm:bdextension} and \ref{thm:extensionDB}.  We represent an element $z_{i}$ in a neighbourhood of the identity in $SO(2)^{r_{i}}$ as $z_{i}=(\underline{z_{i1}}, \cdots, \underline{z_{ir_{i}}})$ with $\underline{z_{ij}}$ given by (\ref{eq:sin}) for $j=1,\cdots, r_{i}$.   Then as an element of $Z_{SO(i)}(h_{i})\subset SO(i)$, $z_{i}$ is block diagonal with $z_{i}=\prod_{j=1}^{r_{i}}\underline{z_{ij}}$, if $i$ is even and $z_{i}= \prod_{j=1}^{r_{i}}\underline{z_{ij}}\times \{1\}$, if $i$ is odd.  Let $z=(z_{1},\cdots, z_{n-1})\in SO(2)\times\cdots\times SO(2)^{r_{n-1}}=SO(2)^{d}$.
  
We compute
$$
\frac{\partial}{\partial z_{ij}^{\prime}} |_{z^{\prime}=0}\,\Ad(z_{2}^{\prime}g_{2,3}^{-1}(z_{2})      \cdots g_{n-2,n-1}^{-1}(z_{n-2}) z_{n-1}^{\prime}\\
      g_{n-2,n-1}(z_{n-2}) \cdots g_{2,3}(z_{2})) \cdot x=
      $$
         \begin{equation}\label{eq:odiff}
      ad(g_{2,3}^{-1}(z_{2})\cdots g_{i-1,i}^{-1}(z_{i-1}) A_{ij} g_{i-1,i}(z_{i-1})\cdots g_{2,3}(z_{2}))\cdot x,
      \end{equation}
for $2\leq i\leq n-1$, $1\leq j\leq r_{i}$.  Here $A_{ij}\in\mathfrak{so}(i)\hookrightarrow\mathfrak{so}(n)$ is a block diagonal matrix with the $j$-th block given by the $2\times 2$ matrix $J=\left[\begin{array}{cc} 
0 &1\\
-1 & 0\end{array}\right]
$ and all other blocks $0$.  Let $\gamma_{i}=g_{i-1,i}(z_{i-1})\cdots g_{3,4}(z_{3}) g_{2,3}(z_{2})\in SO(i)$.  Then equation (\ref{eq:odiff}) implies
\begin{equation}\label{eq:otangent}
T_{x}(\ofibre)=span\{\partial_{x}^{[\gamma_{i}^{-1}A_{ij} \gamma_{i}, x]}|\, 2\leq i\leq n-1,\, 1\leq j\leq r_{i}\}.
\end{equation}
The element $\gamma_{i} \in SO(i)$ conjugates $x_{i}$ into $h_{i}\in\mathcal{D}_{i}\subset\fh_{i}^{reg}$, $\mathcal{D}_{i}$ as in (\ref{eq:fundom}).  Clearly, $\fz_{\fso(i)}(h_{i})$ has basis given by the matrices $A_{ij}$ for $1\leq j\leq r_{i}$.  Hence, the elements $\gamma_{i}^{-1}A_{ij} \gamma_{i}$, $1\leq j\leq r_{i}$ form a basis for $\fz_{\fso(i)}(x_{i})$.  Thus, (\ref{eq:otangent}) gives
$$
T_{x}(\mathfrak{so}(n)_{c})=span\{\partial_{x}^{[\fz_{\fso(i)}(x_{i}), x]} |\,2\leq i\leq n-1\}.  
$$

Now, for $x\in\fso(n)_{\Omega}$ we claim that this is the subspace $V_{x}$.  Indeed, recall equation (\ref{eq:odist}) 
$$
 V_{x}=span\{\partial_{x}^{[\nabla f_{i,j}(x_{i}),x]}|\, 2\leq i\leq n-1, 1\leq j\leq r_{i}\},
$$
where $f_{i,j}$, $1\leq j\leq r_{i}$ generate the ring of Ad-invariant polynomials on $\mathfrak{so}(i)$.  For $x\in\fso(n)_{\Omega}$, $x_{i}$ is regular for all $i$, and therefore the elements $\nabla f_{i,j}(x_{i}),\, 1\leq j\leq r_{i}$ form a basis for the centralizer of $x_{i}$ by Theorem \ref{thm:reg} for any $i$.  Thus, we have 

\begin{equation}\label{eq:integral}
T_{x}(\mathfrak{so}(n)_{c})=span\{\partial_{x}^{[\fz_{\fso(i)}, x]}|\, 2\leq i\leq n-1\}=V_{x}
\end{equation}
It follows from Theorem \ref{thm:ofibre} that $\dim T_{x}(\mathfrak{so}(n)_{c})=\dim \ofibre=\dim SO(2)^{d}=d$.  Thus, for $x\in\mathfrak{so}(n)_{c},\, \dim V_{x}=d$.  Thus, $\ofibre^{sreg}=\ofibre$ for $c\in\Omega_{n}$ by (\ref{eq:sreg}).
  

By Theorem \ref{thm:ofibre} $\ofibre=\ofibre^{sreg}$ is irreducible for $c\in\Omega_{n}$.  It follows immediately from Theorem \ref{thm:cpts} that $\ofibre$ is one $A$-orbit.  The last statement of the theorem follows from (\ref{eq:oaction}). 
\end{proof}


 

 \section{Summary of $A$-orbit structure of $\fgl(n)^{sreg}$}\label{s:summary}
 In this section, we briefly summarize without proof the main results of [Col1], which describe the $A$-orbit structure of all strongly regular elements in $\fgl(n)$.  For complete proofs, we refer the reader to [Col1] or [Col].  In section \ref{ss:moment}, we remarked that for any $c\in\Co^{d+r_{n}}=\Co^{\frac{n(n+1)}{2}}$, $\glsfibre$ is non-empty by Theorem 2.3 in [KW1].  In [Col1], we describe the $A$-orbit structure of $\fgl(n)^{sreg}$ by describing the action of $A$ on the fibres $\glsfibre$ for any $c\in \Co^{\frac{n(n+1)}{2}}$.  
 
	To state the main result, it is more convenient to use a different version of the moment map than the one given in (\ref{eq:map}).   Define a morphism $\Psi: \fgl(n)\to \Co^{\frac{n(n+1)}{2}}$ by 
\begin{equation}\label{eq:newmoment}
	\Psi(x)=(p_{1,1}(x_{1}), p_{2,1}(x_{2}),\cdots, p_{n,n}(x)),
	\end{equation}
	where $p_{i,j}(x_{i})$ is the coefficient of $t^{j-1}$ in the characteristic polynomial of $x_{i}$.  The collection of fibres of the map $\Psi$ is the same as that of the moment map $\Phi$ in (\ref{eq:map}).  It is also easy to see that the action of $A$ preserves the fibres of $\Psi$, and that Theorem \ref{thm:cpts} remains valid when the moment map $\Phi$ in (\ref{eq:map}) is replaced by the map $\Psi$.  For the remainder of the paper, we use the notation $\glfibre$ to denote $\Psi^{-1}(c)$ for $c\in \Co^{\frac{n(n+1)}{2}}$ and $\glsfibre$ to denote $\glfibre\cap\fgl(n)^{sreg}$. 
	
	To describe the fibres $\glfibre$, it is useful to adopt the following convention.  Let $c_{i}\in\C^{i}$ and consider $c=(c_{1}, c_{2},\cdots, c_{n})\in\C^{1}\times\C^{2}\times\cdots\times\C^{n}=\C^{\frac{n(n+1)}{2}}$.  Regard $c_{i}=(z_{1},\cdots, z_{i})$ as the coefficients of the degree $i$ monic polynomial 
 \begin{equation}\label{eq:polyci}
 p_{c_{i}}(t)=z_{1}+z_{2} t+\cdots + z_{i} t^{i-1} +t^{i}.
 \end{equation}
Then $x\in\glfibre$ if and only if $x_{i}$ has characteristic polynomial $p_{c_{i}}(t)$ for all $i$.  
 
The main result is the following, which differs substantially from the generic case studied in section \ref{s:generics}.
 
 \begin{thm}\label{thm:introgen}
 Suppose there are $0\leq j_{i}\leq i$ roots in common between the monic polynomials $p_{c_{i}}(t)$ and $p_{c_{i+1}}(t)$.  Then the number of $A$-orbits in $\glsfibre$ is exactly $2^{\sum_{i=1}^{n-1} j_{i}}.$  For $x\in\glsfibre$, let $Z_{GL(i)}(J_{i})$ denote the centralizer of the Jordan form $J_{i}$ of $x_{i}$ in $GL(i)$.  The orbits of $A$ on $\glsfibre$ are the orbits of a free algebraic action of the complex, commutative, connected algebraic group $Z=\Gprod$ on $\glsfibre$. 
\end{thm}

\begin{rem}
 A very similar result is reached in recent work of Roger Bielawski and Victor Pidstrygach in [BP], which gives interesting geometric interpretations of the work in [KW1] and [KW2].  In [BP], the authors define an action of a group of symplectomorphisms on a space of rational maps of fixed degree from the Riemann sphere into the flag manifold for $GL(n+1)$.   For rational maps of a certain degree, this group is isomorphic to $\Co^{\frac{(n+1)(n)}{2}}$.  They then use information about the orbit structure of this group to get information about the orbit structure of $A$ on $\fgl(n)$.  They also obtain the result that there are $2^{\sum_{i=1}^{n-1} j_{i}}$ orbits in $\glsfibre$.  Our work differs from that of [BP] in that we explicitly list the $A$-orbits in $\glsfibre$ and obtain an algebraic action of $\Gprod$ on $\glsfibre$ whose orbits are the same as those of $A$.
\end{rem}

 The nilfibre $\glsnil$ has particularly interesting structure.  By definition of the map $\Psi$ in (\ref{eq:newmoment}), $x\in\glnil$ if and only if $x_{i}\in\fgl(i)$ is nilpotent for all $i$.  Such matrices have been extensively studied by [Ov] and [PS].  The $A$-orbit structure of $\glsnil$ is a special case of Theorem \ref{thm:introgen}.
 \begin{cor}\label{thm:nil}
On $\glsnil$ the orbits of the group $A$ are given by the orbits of a free algebraic action of the connected, abelian algebraic group $\Gprod\simeq (\Co^{\times})^{n-1}\times \Co^{{n\choose 2}-n+1}$.  There are exactly $2^{n-1}$ $A$-orbits in $\glsnil$.  
\end{cor}

 $A$-orbits in $\glsnil$ determine a certain set of Borel subalgebras that contain the diagonal matrices.  Let $x\in\glsnil$, and let $A\cdot x$ denote its $A$-orbit.  Let $\overline{A\cdot x}$ denote either the Hausdorff or Zariski closure of $A\cdot x$.  These closures agree, since $A\cdot x$ is a constructible set by Proposition \ref{prop:constr} (see [M, pg 60]).  

\begin{thm}
Let $x\in\glsnil$.  Then $ \overline{A\cdot x}$ is the nilradical of a Borel subalgebra $\mathfrak{b}\subset \fgl(n)$ that contains the standard Cartan subalgebra of diagonal matrices in $\fgl(n)$.  
\end{thm} 
\begin{rem}
The strictly lower traingular matrices $\mathfrak{n}^{-}$ and the strictly upper triangular matrices are closures of $A$-orbits in $\glsnil$.  
\end{rem}

Thus, $\overline{A\cdot x}$ is conjugate to $\mathfrak{n}^{-}$ by a unique element of the Weyl group $\sigma\in\mathcal{S}_{n}$, where $\mathcal{S}_{n}$ denotes the symmetric group on $n$ letters.  The permutation $\sigma$ is of the form $\sigma=\sigma_{1}\cdots \sigma_{n-1}$ with $\sigma_{i}=w_{0}^{i}\text{ or } id_{i}$, where $w_{0}^{i}$ is the long element of $\mathcal{S}_{i+1}$ and $id_{i}$ is the identity permutation in $\mathcal{S}_{i+1}$.  For a given nilradical $\overline{A\cdot x}$, $\sigma$ can be determined uniquely using the more detailed description of $\overline{A\cdot x}$ found in [Col1].  We illustrate this with an example.  

\begin{exam}\label{ex:nilexample}
 There is an $A$-orbit in $\fgl(4)_{0}^{sreg}$ whose closure is the nilradical

\begin{equation}\label{eq:nilexample}
\fm=\left[\begin{array}{cccc}
0 & 0 & 0 & a_{1}\\
a_{2} & 0 & 0 & a_{3}\\
a_{4} & a_{5} & 0 & a_{6}\\
0 & 0 & 0 &0\end{array}\right]_{\mbox{\large ,}}
\end{equation}
with $a_{i}\in\Co$ for $1\leq i\leq 6$.  For this example, it is easy to check that the permutation $\sigma=(13)(14)(23)=(1432)$, which is the product of the long elements of $\mathcal{S}_{3}$ and $\mathcal{S}_{4}$, conjugates the strictly lower triangular matrices in $\fgl(4)$ into $\fm$.  
\end{exam}

Theorem \ref{thm:introgen} lets us identify exactly where the action of the group $A$ is transitive on $\glsfibre$.  By Theorem \ref{thm:cpts}, this is equivalent to finding the values of $c\in\Co^{\frac{n(n+1)}{2}}$ for which $\glsfibre$ is irreducible.  Let $\Theta_{n}$ be the set of $c\in\mathbb{C}^{\frac{n(n+1)}{2}}$ such that the monic polynomials $p_{c_{i}}(t)$ and $p_{c_{i+1}}(t)$ are relatively prime (see (\ref{eq:polyci})).  From Remark 2.16 in [KW1], it follows that $\Theta_{n}\subset\Co^{\frac{n(n+1)}{2}}$ is Zariski principal open. 

\begin{cor}\label{c:generic}
The action of $A$ is transitive on $\glsfibre$ if and only if $c\in\Theta_{n}$. 
\end{cor}

  This allows us to find the maximal set of strongly regular elements for which the action of $A$ is transitive on the fibres of the map $\Psi$ in (\ref{eq:newmoment}) over this set.  We can describe the set as follows.  Let

$$
\fgl(n)_{\Theta}=\{x\in\fgl(n) | \;x_{i}\text{ is regular},\; \sigma(x_{i-1})\cap\sigma(x_{i})=\emptyset,\, 2\leq i\leq n-1\},
$$
where $\sigma(x_{i})$ denotes the spectrum of $x_{i}\in\fg_{i}$.  

\begin{thm}
Let $\Psi:\fgl(n)\to\Co^{\frac{n(n+1)}{2}}$ be the map defined in (\ref{eq:newmoment}).  Then $\Psi^{-1}(\Theta_{n})\cap\fgl(n)^{sreg}=\fgl(n)_{\Theta}$.  Thus, the elements of $\fgl(n)_{\Theta}$ are strongly regular.  Moreover, $\fgl(n)_{\Theta}$ is the maximal set of strongly regular elements for which the action of $A$ is transitive on the fibres of $\Psi$.
\end{thm}

\end{document}